\documentclass{article}

\usepackage{amsmath,amssymb,amsfonts}
\usepackage{algorithmic}
\usepackage{graphicx}
\usepackage{textcomp}
\usepackage{subfigure}
\usepackage{url}
\usepackage{authblk}
\newtheorem{remark}{Remark}

\newtheorem{corollary}{Corollary}

\newtheorem{lemma}{Lemma}
\newtheorem{proof}{Proof}
\usepackage{bbm}
\usepackage{float}
\usepackage[printfigures]{figcaps} 
\figcapsoff 

\usepackage[a4paper, total={6in, 8in}]{geometry}
\begin{document}

\title{Fair-MPC: a control-oriented framework for socially just decision-making}
\author[1]{Eugenia Villa}
\author[2]{Valentina Breschi}
\author[1]{Mara Tanelli}
\affil[1]{Department of Electronics, Information and Bioengineering, Politecnico di Milano, Milano, MI 
20133 ITA (e-mail: name.surname@polimi.it) }
\affil[2]{Department of Electrical Engineering, Eindhoven University of Technology, Eindhoven, MB 5600, NL (e-mail: v.breschi@tue.nl)}
\date{}
\maketitle

\begin{abstract}

Control theory can play a pivotal role in tackling many of the global challenges currently affecting our society, representing an actionable tool to help policymakers in shaping our future. At the same time, for this to be possible, elements of social justice must be accounted for within a control theoretical framework, so as not to exacerbate the existing divide in our society. In turn, this requires the formulation of new constraints and control objectives and their integration into existing or new control design strategies. Devising a formally sound framework to ensure social fairness can enable a leap in the comprehension of the meaning of such goals and their non-trivial relation with the usual notion of performance in control.
In this new and challenging context, we propose Fair-MPC, an economic model predictive control scheme to promote fairness in control design and model-based decision-making. Moving beyond traditional measures of performance, Fair-MPC enables the integration of social fairness metrics, articulated in \emph{equality} and \emph{equity}, allowing the design of optimal input sequences to trade-off fairness and efficiency, which are often in conflict. By establishing a theoretical link between Fair-MPC and the standard model predictive control framework (quadratic cost and linear constraints), we provide insights into the impact of social justice objectives on the final control action, which is ultimately assessed along the efficiency of the controller over two numerical examples.
\end{abstract}


\section{Introduction}
\label{sec:introduction}
For a better future, the societal challenges to be faced will be profound and all measures taken to face them will have a decisive impact on our society. It is thus imperative also for the control community to put its knowledge to service, exploiting and conceiving new control theoretical tools to analyze, evaluate, and eventually mitigate the repercussions of the profound social transformations that such changes entail. Among the different opportunities for the control systems community to contribute to the solution of the global challenges our society is currently facing (see \cite{roadmap23}), recent works increasingly strive to investigate social aspects through control-theoretic lenses to achieve socially responsible control and coordination of complex systems. The first context in which this effort has been undertaken is the analysis of opinion formation in networked systems. As an example, the work in \cite{jadbabaie2022inference} investigates the influence of social pressure and inherent opinions on declared ones, having significant implications within legal and political frameworks. A focus on opinions generated from indecision is made in \cite{franci2022breaking}, where it is shown that understanding and modeling factors that influence the final choice between alternatives of (potentially) equal utility can lead to relevant advantages (see e.g., \cite{Stewart2003} and \cite{UN_report_21}). The effect of dynamics norms intended as global-scale trends are introduced in a novel formulation for opinion diffusion in \cite{zino2022facilitating}, enabling policy planners and practitioners to recalibrate their level of intervention to facilitate the achievement of the control objective. 
Societally aware elements are further incorporated into closed-loop strategies in the context of urban traffic and road pricing in, e.g., \cite{othman_analysis_2022,jalota2022creditbased}. According to the recent guidelines indicated in \cite{roadmap23}, it is particularly crucial to ensure that models and strategies that allow for the incorporation of societal drivers in the definition of closed-loop action simultaneously adhere to criteria of social fairness. The latter can be (and it will be here) intended as both equality of outcomes \cite{Hardt2016}, or equality of resources \cite{Arneson1989}. 

The impact of fairness in automated decision-making processes has been analyzed in different works, often leading to opposite conclusions. While \cite{GrgicHlaca2016TheCF} claims that fairness can be achieved through \textquotedblleft unawareness\textquotedblright, \cite{Dwork2011} advocates the paradigm of fairness \textquotedblleft through awareness\textquotedblright, showcasing the benefits of explicitly relying on protected attributes. In the context of multi-agent systems fairness issues are instead often tackled within the reinforcement learning framework, through the \textit{ad-hoc} adaptation of existing techniques. Examples range from the multi-objective deep reinforcement learning problem considered in \cite{WEYMARK1981}, concurrently aiming at a Pareto-optimal (i.e., \emph{efficient}) and \emph{impartial}, namely \emph{equal} and \emph{equitable} policy, to the hierarchical reinforcement learning scheme proposed in \cite{Jiang2019}. Game theoretic perspectives of fairness in the multi-agent context (with conflicting objectives) are provided in \cite{NIPS2014,Behrunani2023DesigningFI}. Despite the growing number of works that approach the issue of fairness in automated decision-making, this is still a problem unresolved for most of its parts  \cite{Chouldechova2018}. A key challenge that emerges from all these works specifically revolves around dynamically embedding fairness into the decision-making process in a quantitative way.
Indeed, as argued in \cite{DAmour2020FairnessIN}, most approaches focused on understanding the impact of unfairness and improving fairness in decision-making look at static or single-step settings only. As well known in the control context, this can have detrimental effects in the long run, with fairness criteria conceived for a static scenario not actually promoting fairness over time but rather jeopardizing it after all (see the discussion in \cite{Liu2018}). Only a few works aim at tackling the dynamic nature of fairness. \cite{Henzinger_2023} introduces techniques tailored for monitoring fairness in real-time, eventually allowing for the mitigation of long-run biases induced by fairness-oriented decision-making but not actively considering it in the design process. Instead, \cite{Creager2020} introduces a new modeling paradigm to achieve dynamic fairness intended as maximum proﬁt, demographic parity, and equal opportunity in decision-making processes. A reinforcement learning algorithm for recommendation capable of dynamically adjusting its policy to ensure that fairness requirements are consistently met is proposed in \cite{ge2021towards}. Again related to recommendation systems, \cite{morik2020controlling} proposes a strategy for fair dynamic ranking, \textit{i.e.,} the equal allocation of exposure based on item relevance. Instead, \cite{pagan2023classification} classifies feedback loops that can result from automated decisions and their impact on algorithmic fairness. 

Despite the richness of the topic, its natural fit for a control-oriented framework, and the need to include it to make control theory an actionable tool for a better future \cite{roadmap23}, there remains a scarcity of methodological works that focus on studying the integration of fairness criteria \textquotedblleft into loop\textquotedblright. 

\paragraph*{Contribution} Prompted by the importance of fairness in decision-making, and by the understanding of how dynamics and feedback play a crucial role in its (mis)propagation over time, in this work we aim to address the following question: \emph{Can control be fair?} To this end, to the best of our knowledge, we make the first attempt to encode two fairness principles (\emph{equality} and \emph{equity}) into feedback control, in a context where several agents have to accomplish an individual task (\emph{i.e.,} reach a desired steady state condition) by satisfying a set of local constraint competing for the same (limited) share of control resources. On the one hand, each system should thus attain its target for the control strategy to be \emph{effective}. Meanwhile, the shared control resources should be \emph{fairly} distributed among the agents so that the designed control policy is \emph{impartial}, while agents should attain their goal \emph{equally well} so that none of them is left behind. The first contribution of this work lies in the formalization of these contrasting objectives within a control-oriented framework formulating a \textquotedblleft fair\textquotedblright \ control design problem. A natural framework to cope with the constrained, multi-objective nature of this problem is that of economic model predictive control (MPC) \cite{ellis2017economic}, whose predictive essence allows to account for long-term fairness-induced effects at design time. Assuming that a centralized controller decides on the inputs fed to each agent and aims not to waste any control resource, our second contribution is thus the introduction of a novel economic MPC scheme (that we call from now on \emph{Fair-MPC}), which incorporates our definitions of performance, equality, and equity for the attainment of a \emph{fair control action}.  Based on our design choices, we can establish a direct link between the proposed Fair-MPC scheme and a traditional model predictive control one, providing insights into the impact of the fairness-induced terms on the final control action. Our formalization also allows us to devise a set of novel compact indexes to assess the level of fairness of the designed control action, along with an adaption strategy for the penalties of the fairness-driven elements in the cost, mitigating long-term undesired effects induced by them on the systems' behavior. The theoretical framework of economic MPC allows us to overcome the limitation of \cite{Wen2019} by shifting from a static to a dynamic framework. Besides, we present extensive numerical results toward answering another question, namely \emph{What is the impact of incorporating fairness principles in control?}. Our results spotlight the changes in the systems' responses due to introducing fairness into the control design logic. At the same time, we show the benefits of introducing an adaptation strategy for the penalties of the predictive problem by accounting for the inherent dynamical nature of fairness itself.  
This analysis shows that pursuing social justice along with more traditional control objectives opens a crucial discussion that must redefine the notion of \textit{optimality}, as potentially having brought all agents closer to their goal, with maybe only a few fully reaching it may have more value than having a good share of them fulfilling it but with exacerbated inequalities.

\paragraph*{Outline} The paper is organized as follows. Section~\ref{sec:setting} formalizes the setting and the main goal of the work. The fair model predictive control framework is introduced in Section~\ref{sec:FMPC}, along with some of its possible extensions, while its properties are discussed in Section~\ref{sec:properties}. A set of indexes to assess the performance of Fair-MPC and auto-tune the penalties of the associated cost are defined in Section~\ref{sec:practical}. The effects of introducing fairness in control and the effectiveness of these tuning strategies are then shown in Section~\ref{sec:examples} on two numerical examples. The paper concludes with some final remarks and directions for future work.
\paragraph*{Notation} Let $\mathbb{N}$, $\mathbb{R}$ and $\mathbb{R}^{+}$  be the set of natural, real and positive real numbers respectively. Moreover, let us indicate with $\mathbb{N}_{0}$ the set of natural numbers including zero. Denote with $\mathbb{R}^{n}$ the set of real column vectors of dimension $n$ and with $\mathbb{R}^{n \times m}$ the set of real matrices with $n$ rows and $m$ columns. The identity matrix of dimension $n\times n$ is indicated as $I_{n}$, while unitary vectors with $n$ columns are denoted as $\mathbbm{1}_{n}$. Given a vector $\mathbf{x} \in \mathbb{R}^{n}$, we denote its $i$-th component as $[\mathbf{x}]_{i}$, for $i=1,\ldots,n$, while we indicate its 1-norm as $\|\mathbf{x}\|_{1}$ and its 2-norm as $\|\mathbf{x}\|_{2}$. Given any matrix $\mathbf{Q} \in \mathbb{R}^{n \times n}$, $\mathbf{Q} \succeq 0$ ($\mathbf{Q} \succ 0$) indicate that it is positive semi-definite (positive definite). Given a set of matrices $\{\mathbf{Q}^{i}\}_{i=1}^{N}$, $\mathrm{diag}(\mathbf{Q}^{1},\ldots,\mathbf{Q}^{N})$ denotes the block-diagonal matrix, whose diagonal blocks correspond to the matrices in such set. Quadratic forms are indicated as $\|\mathbf{x}\|_{\mathbf{Q}}^{2}=\mathbf{x}^\top \mathbf{Q}\mathbf{x}$, with $\mathbf{x}^{\top}$ being the transpose of $\mathbf{x}$. Given a set $\mathcal{F}$, we denote its interior as $\mathrm{int}(\mathcal{F})$.
\section{Setting and goal}\label{sec:setting}

Consider a set of $N$ \emph{stabilizable} dynamical systems, each described by its own \emph{linear time invariant} (LTI) model
\begin{equation}\label{eq:dyn_sys}
\mathbf{x}_{t+1}^{i}=\mathbf{A}^{i}\mathbf{x}_{t}^{i}+\mathbf{B}^{i}\mathbf{u}_{t}^{i},
\end{equation}	
where $\mathbf{x}_{t}^{i} \!\in\! \mathbb{R}^{n}$ and $\mathbf{u}_{t}^{i} \!\in\! \mathbb{R}^{m}$ are the \emph{measurable} state and input of the $i$-th system at time $t \!\in\! \mathbb{N}_{0}$, respectively, while the matrices $A^{i} \!\in\! \mathbb{R}^{n\times n}$ and $B^{i} \!\in\! \mathbb{R}^{n\times m}$ inherently characterize its main dynamical features, with $i=1,\ldots,N$. Note that, despite their differences, we assume the inputs and the states of these systems to have the same \textquotedblleft meaning\textquotedblright, thus sharing the same definition and dimensions for them to be comparable. Suppose that these systems work within the same environment and that their operations are controlled by a \emph{centralized} control unit. Therefore, let us further define their ensemble dynamics as    \begin{subequations}\label{eq:ensemble_dyn}
    \begin{equation} \mathbf{x}_{t+1}=\mathbf{A}\mathbf{x}_{t}+\mathbf{B}\mathbf{u}_{t}=\mathbf{A}\begin{bmatrix}
    \mathbf{x}_{t}^{1}\\    
    \vdots\\
    \mathbf{x}_{t}^{N}
    \end{bmatrix}\!+\!\mathbf{B}\begin{bmatrix}
    \mathbf{u}_{t}^{1}\\
    \vdots\\
    \mathbf{u}_{t}^{N}
    \end{bmatrix},
    \end{equation}
    where $\mathbf{A} \in \mathbb{R}^{nN \times nN}$ and $\mathbf{B} \in \mathbb{R}^{nN \times mN}$ as
    \begin{equation}
\mathbf{A}=\mathrm{diag}\left(\mathbf{A}^{1},\ldots,\mathbf{A}^{N}\right),~~~\mathbf{B}\!=\!\mathrm{diag}\left(\mathbf{B}^{1},\ldots,\mathbf{B}^{N}\right).   
    \end{equation}
    \end{subequations}
Moreover, let us assume that each system pursues an individual goal, here represented by a specific equilibrium state $\mathbf{x}_{s}^{i} \in \mathbb{R}^{n}$ that each system targets, \emph{i.e.,}
\begin{equation}\label{eq:inputequilibrium_def}
\mathbf{x}_{s}^{i}=\mathbf{A}^{i}\mathbf{x}_{s}^{i}+\mathbf{B}^{i}\mathbf{u}_{s}^{i},
\end{equation}
where $\mathbf{u}_{s}^{i} \in \mathbb{R}^{m}$ is the corresponding equilibrium input. Further assume that this goal has to be achieved while the systems' inputs and states are subject to the following (local) polytopic constraints:
\begin{equation}\label{eq:constraints}
	\mathbf{u}_{t}^{i} \in \mathcal{U}^{i},~~~\mathbf{x}_{t}^{i} \in \mathcal{X}^{i}, \forall t \in \mathbb{N}_{0},
\end{equation}
with $\mathcal{U}^{i} \subseteq \mathbb{R}^{m}$, $\mathcal{X}^{i} \subseteq \mathbb{R}^{n}$ defined such that $\mathbf{u}_{s}^{i} \in \mathrm{int}(\mathcal{U}^{i})$ and $\mathbf{x}_{s}^{i} \in \mathrm{int}(\mathcal{X}^{i})$, for $i=1,\ldots,N$. Meanwhile, suppose that systems have to share the same resources to achieve their objectives, with both positive and negative control actions having a non-negative weight on the count of the total effort. In turn, this constraint limits the entity of the local control actions $\mathbf{u}_{t}^{i}$ as follows:
\begin{equation}\label{eq:shared_res}
	\sum_{i=1}^{N}\|\mathbf{u}_{t}^{i}\|_{1} \leq \bar{U}_{t},
\end{equation} 
where $\bar{U}_{t} \in \mathbb{R}^{+}$ is the maximum control effort that the centralized unit can tolerate at time $t$. 

Within this framework, our objective is to devise a \emph{centralized} controller that \emph{optimally} steers the state of each system towards its target, while accounting for some \emph{fairness criteria} when \emph{allocating} the available resources. In particular, we focus on designing control policies that reward \emph{equality}, \emph{i.e.,} attaining an even distribution of the available control resources, and \emph{equity}, namely promoting systems to be comparably close to their individual targets, in spite of their differences.
 
\section{Introducing fairness in MPC}\label{sec:FMPC}

    To attain our goal, we formalize the following centralized \emph{Fair Model Predictive Control} (Fair-MPC) problem:
    \begin{subequations}\label{eq:F-MPC}
	\begin{align}
		& \underset{\tilde{u},\tilde{x},\varepsilon_{u},\varepsilon_{x}}{\min}~J_{L}^{\mathrm{fair}}(\mathbf{\tilde{u}},\mathbf{\tilde{x}},\bar{U}_{t},\mathbf{x}_{s})\!+\!V(\mathbf{\tilde{x}}_{L},\mathbf{\tilde{u}}_{L})+\!\lambda_{u}\varepsilon_{u}^{2}\!+\!\lambda_{x}\varepsilon_{x}^2 \label{eq:FMPC_cost}\\
		&~~~~\mbox{s.t. }~~\mathbf{\tilde{x}}_{k+1}=\mathbf{A}\mathbf{\tilde{x}}_{k}+\mathbf{B}\mathbf{\tilde{u}}_{k},~~~k=0,\ldots,L-1,\label{eq:FMPC_model1}\\
		& \qquad~~~ ~~ \mathbf{\tilde{x}}_{0}=\mathbf{x}_{t},\label{eq:FMPC_initcond}
    \end{align}
  \begin{align}
		&  \qquad ~~~~~ \mathbf{\tilde{u}}_{k} \in \mathcal{U},~~\mathbf{\tilde{x}}_{k} \in \mathcal{X},~~~~~~k=0,\ldots,L,\label{eq:FMPC_indconstr}\\
		& \qquad ~~~~ \sum_{i=1}^{N} \|\mathbf{\tilde{u}}_{k}^{i}\|_{1} \leq \bar{U}_{t},~~~~~~~~~ k=0,\ldots,L,\label{eq:FMPC_alloconstr}\\
        & \qquad~~~ ~~\mathbf{\tilde{x}}_{L}=\mathbf{A}\mathbf{\tilde{x}_{L}}+\mathbf{B}\mathbf{\tilde{u}}_{L},\label{eq:FMPC_terminalsteady}\\ 
        & \qquad~~~~~\|\mathbf{\tilde{x}}_{L}\!-\!\mathbf{x}_{s}\|_{1}\!\leq\! \varepsilon_{x},~~~\|\mathbf{\tilde{u}}_{L}-\mathbf{u}_{s}\|_{1}\!\leq \!\varepsilon_{u},\label{eq:FMPC_terminalbound}\\ 
        &\qquad~~~~~\varepsilon_{x},\varepsilon_{u} \geq 0,
	\end{align}
\end{subequations}
where $L \geq 1$ is the user-defined prediction horizon, $\mathbf{u}_{s}$ and $\mathbf{x}_{s}$ stack the equilibrium inputs and states each system aim at attaining, respectively, the optimization variable $\mathbf{\tilde{u}}$ comprises all predicted inputs over $L$, namely $\mathbf{\tilde{u}}=\{\mathbf{\tilde{u}}_{k}\}_{k=0}^{L}$, while $\mathbf{\tilde{x}}=\{\mathbf{\tilde{x}}_{k}\}_{k=0}^{L}$ is composed by the associated states, predicted from the initial condition $\mathbf{\tilde{x}}_{0}$ according to the ensemble dynamics \eqref{eq:ensemble_dyn}. Since we aim at solving the Fair-MPC problem in a receding horizon fashion at each time step $t$, the constraint in \eqref{eq:FMPC_initcond} sets the initial state of each system over the current optimization window at its actual value. Meanwhile, $\mathcal{U}=\mathcal{U}^{1}\times \cdots \times \mathcal{U}^{N}$ and $\mathcal{X}= \mathcal{X}^{1}\times \cdots \times \mathcal{X}^{N}$, so that \eqref{eq:FMPC_indconstr} and \eqref{eq:FMPC_alloconstr} incorporate all the constraints of our control problem. Inspired by \cite{FAGIANO2013}, we enforce the two additional constraints in \eqref{eq:FMPC_terminalsteady} and \eqref{eq:FMPC_terminalbound} on the terminal states and inputs. Specifically, we impose that an equilibrium state is attained at the end of the prediction horizon, which can be nonetheless different from the targeted one. At the same time, we impose the difference between the desired target (\emph{i.e.,} the inputs and states satisfying \eqref{eq:inputequilibrium_def}) to be limited. Not to over-constraint the problem, this limitation is implemented by introducing two slack variables, whose value is minimized along with the fairness objective and controlled by the tunable parameters\footnote{These parameters should be chosen as large as possible, for the slacks to actually limit the distance between the equilibrium attained at the end of the horizon and the desired one.} $\lambda_{u},\lambda_{x}>0$. Moreover, the terminal cost is defined as the scaled quadratic form: 
\begin{equation}\label{eq:terminal_cost}
V(\mathbf{\tilde{x}}_{L},\mathbf{\tilde{u}}_{L})=\beta\ell^{\mathrm{fair}}(\mathbf{\tilde{x}}_{L},\mathbf{\tilde{u}}_{L}),
\end{equation}
where $\beta>0$ is an additional parameter to be calibrated and 
\begin{equation}\label{eq:FMPC_stagecost}
    \ell^{\mathrm{fair}}(\mathbf{\tilde{x}},\mathbf{u})=\|\mathbf{\tilde{x}}-\mathbf{x}_{s}\|_{\mathbf{\tilde{Q}}}^{2}+\sum_{i=1}^{N}\rho^{i}\left(\|\mathbf{\tilde{u}}^{i}\|_{1}-\frac{\bar{U}_{t}}{N}\right)^2,
\end{equation}
is the stage cost of the considered problem with the relative weights $\tilde{Q}\succeq 0$ and $\rho^{i}\geq 0$.
Instead, since our design objective is three-fold, the cost in \eqref{eq:FMPC_cost} is defined as the combination of three terms, namely:
\begin{equation}\label{eq:cost_divided}
	J_{L}^{\mathrm{fair}}(\mathbf{\tilde{u}},\mathbf{\tilde{x}},\bar{U}_{t},\mathbf{x}_{s})=J_{p}(\mathbf{\tilde{x}},\mathbf{x}_{s})+J_{u}(\mathbf{\tilde{u}},\bar{U}_{t})+J_{e}(\mathbf{\tilde{x}},\mathbf{x}_{s}),
\end{equation}
with $(i)$ $J_{p}(\mathbf{\tilde{x}},\mathbf{x}_{s})$ penalizing poor individual performance, $(ii)$ $J_{u}(\mathbf{\tilde{u}},\bar{U}_{t})$ weighting uneven shares of the available resources (equality) and $(iii)$ $J_{e}(\mathbf{\tilde{x}},\mathbf{x}_{s})$ rewarding comparable performance among the systems (equity). To fully characterize Fair-MPC, we now focus on the description of these three elements of the cost function.  
\subsubsection{Minding individual performance through $J_{p}(\mathbf{\tilde{x}},\mathbf{x}_{s})$} In our framework, each system is associated with an equilibrium point $\mathbf{x}_{s}^{i}$ that it aims at attaining. To account for this individual goal, the first term is chosen as
\begin{equation}\label{eq:performance_cost}
	J_{p}(\mathbf{\tilde{x}},\mathbf{x}_{s})=\sum_{i=1}^{N}\sum_{k=0}^{L-1}\|\mathbf{\tilde{x}}_{k}^{i}-\mathbf{x}_{s}^{i}\|_{\mathbf{Q}^{i}}^{2}.
\end{equation}
This formulation allow us to account for the distance of each system from its individual goal, while considering the level of importance attributed to the achievement of such target by the different systems through the personalized positive definite weight $\mathbf{Q}^{i}$, for $i=1,\ldots,N$. 
\subsubsection{Accounting for equality via $J_{u}(\mathbf{\tilde{u}},\bar{U}_{t})$} 
Differently from what is conventionally done in MPC, the penalty on individual performance in \eqref{eq:performance_cost} does not include a term weighting the distance between the control inputs and associated values at the equilibrium. This choice is mainly motivated by the fact that our goal is to select the optimal inputs mainly focusing toward guaranteeing \emph{equality} among the systems.

To this end, recall that the systems can perform their individual tasks with an input that is limited by the maximum allowable control effort $\bar{U}_{t}$, as formalized in \eqref{eq:FMPC_alloconstr}. For Fair-MPC to be actually fair, the latter should be equally distributed over the $N$ systems, so as to provide all of them with equal means to attain their goals. This requirement is embedded in Fair-MPC through the second term in the cost, which is formalized as follows:
\begin{equation}\label{eq:equality_cost}
	J_{u}(\mathbf{\tilde{u}},\bar{U}_t)=\sum_{i=1}^{N}\sum_{k=0}^{L-1} \rho^{i}\left(\|\mathbf{\tilde{u}}_{k}^{i}\|_{1}-\frac{\bar{U}_{t}}{N}\right)^2,
\end{equation}  
so as to guide the optimization of the local control actions towards the average of the available resources at step $t$, according to the relative importance weights $\rho^{i} \geq 0$. 

\begin{remark}[An outlook on $J_{u}(\mathbf{\tilde{u}},\bar{U}_{t})$]
	By leveraging \eqref{eq:shared_res}, the following relationship holds  
	\begin{equation*}
		\|\mathbf{\tilde{u}}_{k}^{i}\|_{1}-\frac{1}{N}\sum_{j=1}^{N}\|\mathbf{\tilde{u}}_{k}^{j}\|_{1} \geq \|\mathbf{\tilde{u}}_{k}^{i}\|_{1}-\frac{\bar{U}_{t}}{N}.
	\end{equation*}
	As such, $J_{u}(\mathbf{\tilde{u}},\bar{U}_{t})$ in \eqref{eq:equality_cost} allows one to implicitly minimize the difference between the individual control effort and the average effort over the whole set of systems. 
\end{remark}
\begin{remark}[Equality and optimal action]
    Even if we do not introduce a penalty on the distance between the optimal input and the equilibrium $\mathbf{u}_{s}$, the impact of the equality loss in \eqref{eq:equality_cost} can still be linked to the effect of such a penalty, as shown in Section~\ref{sec:properties}.
\end{remark}
\subsubsection{Incorporating equity through $J_{e}(\mathbf{\tilde{x}},\mathbf{x}_{s})$} For the optimal control action to be equitable, it should allow all the $N$ systems to comparably attain their individual goals. Given our assumptions, dissimilarities in the performance of the systems can be penalized by looking at the distance of each system from the desired equilibrium point $\mathbf{x}_{s}^{i}$. As a consequence, the third term in \eqref{eq:FMPC_cost} is formalized as follows:
\begin{equation}\label{eq:equity_cost}
	J_{e}(\mathbf{\tilde{x}},\mathbf{x}_{s})=\sum_{i=1}^{N}\sum_{k=0}^{L-1}\bigg\| \mathbf{\tilde{e}}_{k}^{i}-\frac{1}{N}\sum_{j=1}^{N}\mathbf{\tilde{e}}_{k}^{j}\bigg\|_{\mathbf{W}^{i}}^{2},
\end{equation}
where the individual weight $\mathbf{W}^{i} \succeq 0$ provides a loose indication on the local importance of enforcing equity, with $i=1,\ldots,N$, and $\mathbf{\tilde{e}}_{k}^{i}$ is the distance of the $i$-th system form the desired equilibrium at the $k$-th instant of the prediction horizon, \emph{i.e.,}
\begin{equation}
\mathbf{\tilde{e}}_{k}^{i}=\mathbf{\tilde{x}}_{k}^{i}-\mathbf{x}_{s}^{i},~~~~ i\!=\!1,\ldots,N,~~k\!=\!0,\ldots,L-1. 
 \end{equation}
Clearly, the aim of this third term of \eqref{eq:FMPC_cost} is to steer the systems to be equally close to complete their individual tasks throughout the prediction horizon.
\begin{remark}[On the Fair-MPC cost]
	Based on the definitions in \eqref{eq:performance_cost}-\eqref{eq:equity_cost}, the overall loss $\nonumber J_{L}^{\mathrm{fair}}(\mathbf{\tilde{u}},\mathbf{\tilde{x}},\bar{U}_{t},\mathbf{x}_{s})$ in \eqref{eq:cost_divided} depends on three user-defined weights $\{\mathbf{Q}^{i},\rho^{i},\mathbf{W}^{i}\}_{i=1}^{N}$. The latter should be selected to trade-off between the different objectives of Fair-MPC, while scaling the terms in the cost to the same order of magnitude. Accordingly, we can rewrite $\rho^{i}$ and $\mathbf{W}^{i}$ as  
	\begin{subequations}\label{eq:scaled_weights}
	\begin{align}
		& \rho^{i}=\gamma_{u}\bar{\rho}^{i},\\
		& \mathbf{W}^{i}=\mathbf{\Gamma}_{e}\mathbf{\bar{W}}^{i},
	\end{align}
	\end{subequations}
	where $\gamma_{u}$ and $\mathbf{\Gamma}_{e}$ are the scaling factors and $\bar{\rho}^i$ and $\mathbf{\bar{W}}^{i}$ are the actual penalties on the relative importance of equity and equality with respect to tracking.   
\end{remark}
\subsection{A class based reformulation of Fair-MPC}
Assume that the $N$ systems characterizing the considered framework can be divided into $C$ classes based on their individual features, \emph{e.g.,} they share the same dynamics or the initial distance from their individual targets. Thanks to the flexibility of the Fair-MPC formulation, the differences between these subgroups of systems can be directly embedded within the objective function in \eqref{eq:FMPC_cost} by modifying the weights associated to the fairness-oriented terms in the cost. This leads to the following class-based reformulation of \eqref{eq:equality_cost} and \eqref{eq:equity_cost}:
\begin{subequations}\label{eq:class_based}
	\begin{equation} J_{u}(\mathbf{\tilde{u}},\bar{U}_{t})\!=\!\!\sum_{c=1}^{C}J_{u}^{c}(\mathbf{\tilde{u}},\bar{U}_{t}),~~J_{e}(\mathbf{\tilde{x}},\mathbf{x}_{s})=\!\sum_{c=1}^{C}J_{e}^{c}(\mathbf{\tilde{x}},\mathbf{x}_{s})
	\end{equation}
	where
		\begin{align}
		& J_{u}^{c}(\mathbf{\tilde{u}},\bar{U}_{t})=\sum_{i \in \mathcal{C}^{c}}\sum_{k=0}^{L-1}\rho^{c}\left(\|\mathbf{\tilde{u}}_{k}^{i}\|_{1}-\frac{\bar{U}_{t}}{N}\right)^2, \label{eq:equality_costC}\\
		& J_{e}^{c}(\mathbf{\tilde{x}},\mathbf{x}_{s})=\sum_{i \in \mathcal{C}^{c}}\sum_{k=0}^{L-1} \bigg\| \mathbf{\tilde{e}}_{k}^{i}-\frac{1}{N}\sum_{j=1}^{N}\mathbf{\tilde{e}}_{k}^{j}\bigg\|_{\mathbf{W}^{c}}^{2}, \label{equity_costC}
	\end{align}
	while $\mathcal{C}^{c}$ indicates the $c$-th cluster of systems, and the weights $\rho^{c}$ and $\mathbf{W}^{c}$ are now class-dependent (rather than system-dependent) penalties, for $c=1,\ldots,C$.
\end{subequations}  
\subsection{On the allocation constraint}
The Fair-MPC formulation in \eqref{eq:F-MPC} allows us to consider two different scenarios with respect to the allocation constraint in \eqref{eq:FMPC_alloconstr}. Indeed, the problem can be solved when the bound on the control effort is constant at all times, \emph{i.e.,}
\begin{equation}\label{eq:input_1}
	\bar{U}_{t}=\bar{U},~~\forall t,
\end{equation} 
or when it adapts over time according to the effort required at the previous time instants, namely
\begin{equation}\label{eq:input_2}
	\begin{cases}
		\bar{U}_{t}=\bar{U}, \mbox{ if } t=0,\\
		\bar{U}_{t}=\bar{U}_{t-1}-\sum_{i=1}^{N}\|\mathbf{u}_{t-1}^{i}\|_{1}, \mbox{ otherwise}.
	\end{cases}
\end{equation}
These two choices entail the availability of either an unlimited or a limited amount of control resources over time, respectively. Indeed, when the scenario described by \eqref{eq:input_1} is considered, the limitless resources available are uniformly issued over time. On the other hand, \eqref{eq:input_2} accounts for the consumption of the available resources at previous time instants, leading to the progressive exhaustion of $\bar{U}$. At the same time, note that \eqref{eq:FMPC_alloconstr} does not allow us to explicitly account for the consumption of control resources when designing the input. The latter can be considered replacing \eqref{eq:FMPC_alloconstr} with the constraint
\begin{subequations}\label{eq:FMPC_alloconstr2}
\begin{equation}
    \sum_{i=1}^{N} \|\mathbf{\tilde{u}}_{k}^{i}\|_{1} \leq \tilde{U}_{k},
\end{equation}
and introducing the dynamics of resource consumption, \emph{i.e.,}
\begin{equation}
    \tilde{U}_{k+1}=\tilde{U}_{k}-\sum_{i=1}^{N}\|\mathbf{\tilde{u}}_{k}^{i}\|_{1},~~~\tilde{U}_{0}=\bar{U}_{t},
\end{equation}
with $\tilde{U}_{k}$ becoming an additional optimization variable of Fair-MPC, for $k=0,\ldots,L-1$.
\end{subequations}

\section{Properties of Fair-MPC}\label{sec:properties}
In this section, we analyze the properties of the Fair-MPC scheme in \eqref{eq:F-MPC}, by initially focusing on its component $J_{e}(\mathbf{\tilde{x}},\mathbf{x}_{s})$ enforcing equity. As formalized in the following Lemma, this term of the cost simply acts as a change on the penalty characterizing the performance loss in \eqref{eq:performance_cost}. 
\begin{lemma}[Impact of equality]\label{lemma:equality}
    Let $\mathbf{S}^{i} \in \mathbb{R}^{n \times nN}$ be a matrix defined as follows:
    \begin{subequations}\label{eq:S_def}
    \begin{equation}
        \mathbf{S}^{i}=\begin{bmatrix}
            \mathbf{S}_{1}^{i} & \mathbf{S}_{2}^{i} & \cdots & \mathbf{S}_{N}^{i}
        \end{bmatrix},
    \end{equation}
    with
    \begin{equation}
        \mathbf{S}_{j}^{i}=\begin{cases}
        \frac{N-1}{N}I_{n} \mbox{ if } i=j,\\
        -\frac{1}{N}I_{n} \mbox{ if } i\neq j,
        \end{cases}~~\forall i,j \in \{1,\ldots,N\}.
    \end{equation}
    \end{subequations}
    Then, the overall cost in \eqref{eq:cost_divided} can be recast as
    \begin{subequations}\label{eq:cost_pure_tracking_plus_equality}
    \begin{equation}
    J_{L}^{\mathrm{fair}}(\mathbf{\tilde{u}},\mathbf{\tilde{x}},\bar{U}_{t},\mathbf{x}_{s})=\tilde{J}_{p}(\mathbf{\tilde{x}},\mathbf{x}_{s})+J_{u}(\mathbf{\tilde{u}},\bar{U}_{t}),
    \end{equation}
    where the equivalent performance loss $\tilde{J}_{p}(\mathbf{\tilde{x}},\mathbf{x}_{s})$ is given by
    \begin{equation}
        \tilde{J}_{p}(\mathbf{\tilde{x}},\mathbf{x}_{s})=\sum_{k=0}^{L-1}\|\mathbf{\tilde{x}}_{k}-\mathbf{x}_{s}\|_{\tilde{\mathbf{Q}}}^{2},
    \end{equation}
    with
    \begin{equation}\label{eq:new_Qweight}
        \tilde{\mathbf{Q}}=\mathrm{diag}(\mathbf{Q}^{1},\ldots,\mathbf{Q}^{N})+\sum_{i=1}^{N}(\mathbf{S}^{i})^{\top}\mathbf{W}^{i}\mathbf{S}^{i} \succ 0,
    \end{equation}
    and $\{\mathbf{Q}^{i}\}_{i=1}^{N}$ being the penalties characterizing the performance-oriented cost in \eqref{eq:performance_cost}.
    \end{subequations}
\end{lemma}
\begin{proof}
    Consider the $i$-th term of \eqref{eq:equity_cost}, namely
    \begin{equation*}
        \left\|\mathbf{e}_{k}^{i}-\frac{1}{N}\sum_{j=1}^{N}\mathbf{\tilde{e}}_{j}\right\|_{\mathbf{W}^{i}}^{2}\!\!\!\!\!\!=\left\|(\mathbf{\tilde{x}}_{k}^{i}-\mathbf{x}_{s}^{i})-\frac{1}{N}\sum_{j=1}^{N}(\mathbf{\tilde{x}}_{k}^{j}-\mathbf{x}_{s}^{j})\right\|_{\mathbf{W}^{i}}^{2}\!\!\!\!\!\!\!.
    \end{equation*}
    By grouping all terms depending on the difference $\mathbf{\tilde{x}}_{k}^{i}-\mathbf{\tilde{x}}_{s}^{i}$, the latter can be equivalently recast as
    \begin{align}
    \nonumber \left\|\mathbf{e}_{k}^{i}-\frac{1}{N}\sum_{j=1}^{N}\mathbf{\tilde{e}}_{j}\right\|_{\mathbf{W}^{i}}^{2}\!\!\!\!\!&=\|\mathbf{S}^{i}(\mathbf{\tilde{x}}_{k}-\mathbf{x}_{s})\|_{\mathbf{W}^{i}}^{2}\\
    &=\|\mathbf{\tilde{x}}_{k}-\mathbf{x}_{s}\|_{(\mathbf{S}^{i})^{\top}\mathbf{W}^{i}\mathbf{S}^{i}}^{2},
    \end{align}
    where $\mathbf{S}^{i}$ is defended as in \eqref{eq:S_def}. The result in \eqref{eq:cost_pure_tracking_plus_equality} follows straightforwardly. Note that, the block matrix $\mathbf{Q}$ is positive definite by construction. Meanwhile, the second term of $\mathbf{\tilde{Q}}$ is positive semi-definite as sum of positive semi-definite quadratic forms. Therefore $\mathbf{\tilde{Q}}$ is \eqref{eq:new_Qweight} is positive definite, thus concluding the proof.
\end{proof}
As expected and, as implicitly chosen by design, equity directly influences the component of the cost that weights individual performance by changing its penalty to promote a fairer balance between individual and ensemble target achievements. In turn, this implies that, besides the term promoting equality, the Fair-MPC cost is purely performance-oriented. Given this result, we now look closely at the equality term (see \eqref{eq:equality_cost}), toward establishing a relationship between the Fair-MPC cost in \eqref{eq:cost_divided} when $\rho^{i}>0$ and the standard loss of MPC, \emph{i.e.,}
\begin{equation}\label{eq:MPC_cost}
J_{L}(\mathbf{\tilde{u}},\mathbf{\tilde{x}},\mathbf{u}_{s},\mathbf{x}_{s})\!=\!\sum_{i=1}^{N}\!\sum_{k=0}^{L-1}\!\left[\|\mathbf{\tilde{x}}_{k}^{i}\!-\!\mathbf{x}_{s}^{i}\|_{\mathbf{\tilde{Q}}^{i}}^{2}\!+\!\|\mathbf{\tilde{u}}_{k}^{i}\!-\!\mathbf{u}_{s}^{i}\|_{\mathbf{{R}}^{i}}^{2}\right]\!,
\end{equation}
where $\mathbf{\tilde{Q}}^{i} \in \mathbb{R}^{n \times n}$ is the $i$-th block of $\mathbf{\tilde{Q}}$ defined as in \eqref{eq:new_Qweight}, while  $\mathbf{{R}}^{i} \in \mathbb{R}^{m \times m}$ only assumed for the moment to be positive definite, for all $i=1,\ldots,N$. We can now formalize the link between MPC with quadratic cost and Fair-MPC. 
\begin{lemma}[MPC \emph{vs} Fair-MPC]\label{lemma:fMPCvsMPC}
    Let $\mathbf{{R}}_{i}$ in \eqref{eq:MPC_cost} be 
    \begin{equation}\label{eq:R_def}
        \mathbf{{R}}^{i}=m\rho^{i}I_{m}\succ 0,
    \end{equation}
    with $\rho^{i}>0$ being the equality penalty in \eqref{eq:equality_cost} and $m$ being the input dimension. Then
    \begin{subequations}\label{eq:inequality}
    \begin{equation}
    J_{L}^{\mathrm{fair}}(\mathbf{\tilde{u}},\mathbf{\tilde{x}},\bar{U}_{t},\mathbf{x}_{s})\leq J_{L}(\mathbf{\tilde{u}},\mathbf{\tilde{x}},\mathbf{u}_{s},\mathbf{x}_{s})\!+\!\Delta(\mathbf{u}_{s},\bar{U}_{t}),~\forall \mathbf{\tilde{u}},\mathbf{\tilde{x}},
    \end{equation}
    with
    \begin{equation}
\Delta(\mathbf{u}_{s},\bar{U}_{t})=mL\sum_{i=1}^{N}\rho^{i}\left\|\mathbf{u}_{s}^{i}-\frac{\bar{U}_{t}}{mN}\mathbbm{1}_{m} \right\|_{2}^{2}.
    \end{equation}
    \end{subequations}
\end{lemma}
\begin{proof}
    Let us consider the $i$-th term of \eqref{eq:equality_cost}, which in turn verifies
    \begin{align*}
        &\|\mathbf{\tilde{u}}_{k}\|_{1}\!-\!\frac{\bar{U}_{t}}{N}\leq\! \sqrt{m}\left(\|\mathbf{\tilde{u}}_{k}\|_{2}\!-\!\frac{\bar{U}_{t}}{\sqrt{m}N} \right) \!\leq \!\sqrt{m}\left|\|\mathbf{\tilde{u}}_{k}\|_{2}\!-\!\frac{\bar{U}_{t}}{\sqrt{m}N} \right|\!\\
        & \qquad \quad = \sqrt{m}\left|\|\mathbf{\tilde{u}}_{k}\|_{2}\!-\!\frac{\bar{U}_{t}}{mN}\|\mathbbm{1}_{m}\|_{2} \right| \leq \sqrt{m}\left\|\mathbf{\tilde{u}}_{k}^{i}\!-\!\frac{\bar{U}_{t}}{mN}\mathbbm{1}_{m} \right\|_{2}\!\!
    \end{align*}
    thanks to the properties of norms and the reverse triangle inequality. Based on the previous upper-bound and the Cauchy-Schwarz inequality, we can further show that
    \begin{align*}
&\rho^{i}\left(\!\|\mathbf{\tilde{u}}_{k}^{i}\|_{1}\!-\!\frac{\bar{U}_{t}}{N}\right)^{\!2}\!\!\leq\!\rho^{i} m \left(\left\|\mathbf{\tilde{u}}_{k}^{i}\!-\!\mathbf{u}_{s}^{i}\!+\!\mathbf{u}_{s}^{i}\!-\!\frac{\bar{U}_{t}}{mN}\mathbbm{1}_{m}\right\|_{2}^{2}\right)\\
&\qquad \qquad \quad \leq m\rho^{i}\left(\|\mathbf{\tilde{u}}_{k}^{i}\!-\!\mathbf{u}_{s}^{i}\|_{2}^{2}\!+\!\left\|\mathbf{u}_{s}^{i}\!-\!\frac{\bar{U}_{t}}{mN}\mathbbm{1}_{m}\right\|_{2}^{2}\right)\!.
    \end{align*}
    The result in \eqref{eq:inequality} easily follows, by noticing that the second term on the right-hand-side of the previous inequality is independent from the time step $k$.
\end{proof}
As a consequence of this lemma, Fair-MPC leads to an optimal control sequence that minimizes the standard MPC cost plus an additional term, provided that the two share the same terminal cost. The additional term is proportional to the distance between the inputs $\mathbf{u}_{s}^{i}$ at the equilibrium and the (normalized) resources one would allocate for the policy to be perfectly equal, for $i=1,\ldots,N$. This property allows us to formalize an additional result.
\begin{corollary}[Resources and optimal actions]
Assume that the $N$ systems described by \eqref{eq:dyn_sys} aim at attaining the same equilibrium point, namely $\mathbf{u}_{s}^{i}=\mathbf{\bar{u}}_{s}$ for all $i=1,\ldots,N$. When the bound on the control effort is selected based on such set point to satisfy
\begin{equation}  \bar{U}_{t}\mathbf{1}_{m}=mN\mathbf{\bar{u}}_{s},~~\forall t, 
\end{equation}
then the optimal sequence resulting from \eqref{eq:F-MPC} is also the optimal one for the MPC scheme with objective in \eqref{eq:MPC_cost}, terminal cost in \eqref{eq:terminal_cost}, and constraints equal to those of the Fair-MPC problem.
\end{corollary}
\begin{proof}
    The proof straightforwardly follows from \eqref{eq:inequality} and the fact that $\Delta(\mathbf{u}_{s},\bar{U}_{t})=0$. Therefore, it is omitted.
\end{proof}
Hence, when the designer has the freedom to choose the maximum control effort and all systems aim at the same target, the optimal control sequence obtained with Fair-MPC also optimizes the standard MPC loss in \eqref{eq:MPC_cost}, despite not explicitly introducing a penalty on the distance between the predicted input and its desired value at the targeted equilibrium. 

We now delve deeper into our choices for the terminal constraints, by recalling that the condition imposed in \eqref{eq:FMPC_terminalsteady} corresponds to one of the main ingredients of the generalized constraint proposed in \cite{FAGIANO2013}. The main difference between our formulation and the one proposed in \cite{FAGIANO2013} lays in the condition \eqref{eq:FMPC_terminalbound}, which replaces the constraint
\begin{subequations}\label{eq:generalized_constraint}
\begin{equation}\label{eq:generalized_constraint2}
    \ell^{\mathrm{fair}}(\mathbf{\tilde{x}}_{L},\mathbf{\tilde{u}}_{L}) \leq \bar{\ell}^{\mathrm{fair}}(t),
\end{equation}
where $\bar{\ell}^{\mathrm{fair}}(t)$ satisfies
\begin{equation}\label{eq:lbar_condition}
    \ell^{\mathrm{fair}}(\mathbf{x}_{s},\mathbf{u}_{s})\leq \bar{\ell}^{\mathrm{fair}}(t).
\end{equation}
\end{subequations}
Nonetheless, in the following we will prove that \eqref{eq:FMPC_terminalbound} (along with the conditions and penalties on the slack variables $\varepsilon_{u}$ and $\varepsilon_{x}$) implicitly guarantees the satisfaction \eqref{eq:generalized_constraint2} for a specific choice of the upper-bound $\bar{\ell}^{\mathrm{fair}}(t)$. To this end, let us define the stage costs characterizing the MPC loss in \eqref{eq:MPC_cost} as
\begin{equation}
\ell(\mathbf{\tilde{x}},\mathbf{\tilde{u}})=\|\mathbf{\tilde{x}}-\mathbf{x}_{s}\|_{\mathbf{\tilde{Q}}}^{2}+\|\mathbf{\tilde{u}}-\mathbf{u}_{s}\|_{\mathbf{R}}^{2},\label{eq:MPC_stagecost}
\end{equation}
with $\mathbf{\tilde{Q}}$ given by \eqref{eq:new_Qweight} and $\mathbf{{R}}$ defined as in \eqref{eq:R_def}. For the sake of showing the connection between the chosen terminal ingredients and the ones in \eqref{eq:generalized_constraint}, we now formalize the relationship between the fair stage cost in \eqref{eq:FMPC_stagecost} and the one in \eqref{eq:MPC_stagecost} at the end of the horizon.
\begin{lemma}[Fair-MPC \emph{vs} MPC (part II)]
Let the terminal stage costs $\ell(\mathbf{\tilde{x}}_{L},\mathbf{\tilde{u}}_{L})$ and $\ell^{\mathrm{fair}}(\mathbf{\tilde{x}}_{L},\mathbf{\tilde{u}}_{L})$ be defined as in \eqref{eq:MPC_stagecost} and \eqref{eq:FMPC_stagecost}, respectively. Then, for any feasible solution of problem \eqref{eq:F-MPC} the following holds:
    \begin{equation}\label{eq:stage_inequality}
        \ell^{\mathrm{fair}}(\mathbf{\tilde{x}}_{L},\mathbf{\tilde{u}}_{L})\leq \ell(\mathbf{\tilde{x}}_{L},\mathbf{\tilde{u}}_{L})+\frac{N^{2}-1}{N^{2}}\sum_{i=1}^{N}\rho^{i}\bar{U}_{t}^{2},
    \end{equation}
    where $\bar{U}_{t} \geq 0$ is the maximum amount of control resources at time $t$ (see \eqref{eq:FMPC_alloconstr}) and $\{\rho^{i}\}_{i=1}^{N}$ are the positive penalties of the equality enforcing loss in \eqref{eq:equality_cost}. 
\end{lemma}
\begin{proof}
    By adding and subtracting to $\ell^{\mathrm{fair}}(\mathbf{\tilde{x}}_{L},\mathbf{\tilde{u}}_{L})$ the penalty on the deviation of inputs from the desired equilibrium value, it straightforwardly follows that
    \begin{equation}\label{eq:fair_gamma}
        \ell^{\mathrm{fair}}(\mathbf{\tilde{x}}_{L},\mathbf{\tilde{u}}_{L})=\ell(\mathbf{\tilde{x}}_{L},\mathbf{\tilde{u}}_{L})+\gamma({\tilde{u}}_{L},\mathbf{u}_{s},\bar{U}_{t}),
    \end{equation}
    where
    \begin{equation*}
    \gamma({\tilde{u}}_{L},\mathbf{u}_{s},\bar{U}_{t})\!=\!\!\sum_{i=1}^{N}\rho^{i}\!\!\left[\!\left(\|\mathbf{\tilde{u}}_{L}^{i}\|_{1}\!-\!\frac{\bar{U}_{t}}{N}\right)^{2}\!\!\!\!-\!m\|\mathbf{\tilde{u}}_{L}^{i}\!-\!\mathbf{u}_{s}\|_{2}^{2}\right]\!.
    \end{equation*}
    In turn, it is straightforward to see that the previous quantity can be upper-bounded as
        \begin{align}\label{eq:gamma_bound}
    \nonumber\gamma({\tilde{u}}_{L},\mathbf{u}_{s},\bar{U}_{t})&\leq\sum_{i=1}^{N}\rho^{i}\left(\|\mathbf{\tilde{u}}_{L}^{i}\|_{1}\!-\!\frac{\bar{U}_{t}}{N}\right)^{2}\\
   \nonumber &\leq \sum_{i=1}^{N}\rho^{i}\left(\|\mathbf{\tilde{u}}_{L}^{i}\|_{1}^{2}+\frac{\bar{U}_{t}^{2}}{N^{2}}\right)\\
    & \leq \frac{N^{2}+1}{N^{2}}\sum_{i=1}^{N}\rho^{i}\bar{U}_{t}^{2}
    \end{align}
    where the first bound is due to the fact the second set of the terms on the right-hand-side of the definition of $\gamma({\tilde{u}}_{L},\mathbf{u}_{s},\bar{U}_{t})$ is non-negative by construction, while the second one is a direct consequence of the non-negativity of the norms and of $\bar{U}_{t}$. The third bound is instead derived by the allocation constraint in \eqref{eq:FMPC_alloconstr}. The proof straightforwardly follows by combining \eqref{eq:fair_gamma} and \eqref{eq:gamma_bound}. 
\end{proof}
This inequality is not surprising, given the result in Lemma~\ref{lemma:fMPCvsMPC}. At the same time, it guarantees that any $\bar{\ell}^{\mathrm{fair}}(t)$ satisfying \eqref{eq:lbar_condition} and verifying
\begin{equation}
    \ell(\mathbf{\tilde{x}},\mathbf{\tilde{u}})+\frac{N^{2}+1}{N^{2}}\sum_{i=1}^{N}\rho^{i}\bar{U}_{t}^{2} \leq \bar{\ell}^{\mathrm{fair}}(t),
\end{equation}
allows for \eqref{eq:generalized_constraint2} to hold. Meanwhile, since the upper-bound in \eqref{eq:stage_inequality} holds for all states and inputs provided that they are feasible, it is easy to see that any 
\begin{equation}\label{eq:l_bardef}
\bar{\ell}^{\mathrm{fair}}(t)=\frac{N^{2}+1}{N^{2}}\sum_{i=1}^{N}\rho^{i}\bar{U}_{t}^{2}+\varepsilon,~~ \forall \varepsilon \geq 0,
\end{equation}
satisfies \eqref{eq:lbar_condition}, thus being a suitable candidate to impose \eqref{eq:generalized_constraint2}. For \eqref{eq:generalized_constraint2} to hold for the previous choice of $\bar{\ell}^{\mathrm{fair}}(t)$, we thus require that
\begin{equation}
    \ell(\mathbf{\tilde{x}}_{L},\mathbf{\tilde{u}}_{L})\leq \varepsilon,~~\varepsilon \geq 0,
\end{equation}
by relying on \eqref{eq:stage_inequality}. The previous condition can be readily translated into two constraints on the Euclidean distances between the terminal states/inputs and their target values
\begin{equation*}
    \|\mathbf{\tilde{x}}_{L}-\mathbf{x}_{s}\|_{2}\leq {\varepsilon}_{x},~~\|\mathbf{\tilde{u}}_{L}-\mathbf{u}_{s}\|_{2}\leq {\varepsilon}_{u},~~\varepsilon_{x},\varepsilon_{u}\geq 0,
\end{equation*}
which holds if the inequalities in \eqref{eq:FMPC_terminalbound} are satisfied thanks to the properties of norms, ultimately motivating our Fair-MPC formulation. We demand additional analysis on the properties deriving by this choice of terminal ingredients to future works.
\begin{remark}[Unfeasible targets]
When the desired set point does not comply with \eqref{eq:constraints}-\eqref{eq:shared_res}, then the bound in \eqref{eq:stage_inequality} becomes 
\begin{equation*}
  \ell^{\mathrm{fair}}(\mathbf{\tilde{x}}_{L},\mathbf{\tilde{u}}_{L})\leq \ell(\mathbf{\tilde{x}}_{L},\mathbf{\tilde{u}}_{L})+\sum_{i=1}^{N}\rho^{i}\left(\|\mathbf{\tilde{u}}_{L}^{i}\|_{1}^{2}+\frac{\bar{U}_{t}^{2}}{N^{2}}\right),    
\end{equation*}
and, thus 
\begin{equation}\label{eq:bound_unfeasible}
    \bar{\ell}^{\mathrm{fair}}(t)=\sum_{i=1}^{N}\rho^{i}\left(\|\mathbf{u}_{s}^{i}\|_{1}^{2}+\frac{\bar{U}_{t}^{2}}{N^{2}}\right)+\varepsilon,
\end{equation}
should hold for \eqref{eq:lbar_condition} to be satisfied. Since the last term in the new bound is positive definite, note that \eqref{eq:bound_unfeasible} is nothing but a specific case of \eqref{eq:l_bardef}, where one has an additional constraint on the magnitude of $\varepsilon$ (other than $\varepsilon \geq 0$).
\end{remark}
\section{Practical implementation of Fair-MPC}\label{sec:practical}
	As in standard MPC, the main tuning knobs of Fair-MPC are the weights $\mathbf{Q}^{i}$, $\rho^{i}$ and $\mathbf{P}^{i}$ characterizing its cost in \eqref{eq:cost_divided}, for $i=1,\ldots,N$. To select weights that are compliant with the relative importance of fairness with respect to individual performance, it is thus crucial to quantify the level of justice achieved through Fair-MPC. To this end, in this section we introduce a set of \emph{key performance indicators}, which can be useful to detect ineffective control actions in practice. These indexes are subsequently exploited to propose an automatic tuning strategy for the weights, allowing one to practically reduce the burden of the time-consuming hyper-parameters calibration phase while accounting for the time-varying nature of fairness (see \emph{e.g.,} \cite{DAmour2020FairnessIN,Liu2018}).
 	\subsection{Performance evaluation}\label{Subsec:performance_indexes}
 	The aim of Fair-MPC is to allow a group of systems to achieve their individual control goals, while accounting for the \emph{fairness} of the chosen control action. To understand if such a balance has been attained, it is thus fundamental to define a set of \emph{indicators} providing a quantitative assessment of these objectives. In what follows, we introduce the indexes we propose to assess the \emph{fairness} of Fair-MPC actions.
 	\subsubsection{Evaluating individual performance} 
 	We first introduce two indicators, respectively embedding the capabilities of Fair-MPC to allow each system to attain its target, and its effects on the time required for the individual tasks completion. To understand whether the control action steers the individuals to their final objective, we define the averaged indicator:
 	\begin{equation}\label{eq:tracking_index1}
 		\mathcal{H}_{s}=\frac{1}{T}\sum_{t=0}^{T-1}e^{-\frac{1}{N}\sum_{i=1}^{N}\|\mathbf{x}_{s}^{i}-\mathbf{x}_{t}^{i}\|_{2}} \in [0,1],
 	\end{equation}    
 	whose value is $1$ when all systems immediately reach their individual goals, while it decreases towards zero according to the average distance of the systems with respect to their target over time. Note that, this index can be easily defined for either a cluster of systems or a single individual, by properly manipulating the exponent in \eqref{eq:tracking_index1}. 
    
    Now, let $\tau^{i}$ be the average number of time instants required for the distance between the $i$-th system and the desired equilibrium to be at a pre-fixed percentage $\alpha_{\%}$ of its initial value, \emph{i.e.,}
    \begin{equation}\label{eq:tau}
 		\|\mathbf{x}_{s}^{i}-\mathbf{x}_{\tau^{i}}^{i}\|_{2} \leq \frac{\alpha_{\%}}{100} \|\mathbf{x}_{s}^{i}-\mathbf{x}_{0}^{i}\|_{2}.
 	\end{equation} 
    The performance index $	\mathcal{H}_{\tau} \in [0,1]$ associated with the average time required for the accomplishment of individual tasks is:   
 	\begin{equation}\label{eq:tracking_time}
 		\mathcal{H}_{\tau}=1-\frac{1}{T-1}\left(\frac{1}{N}\sum_{i=1}^{N} \tau^{i}\right).
 	\end{equation}
 	Note that, $\mathcal{H}_{\tau}$ is equal to $1$ only when the initial outputs correspond to the equilibrium, while it tends to $0$ as soon as the time to complete individual tasks approaches $T-1$ steps. 
    	\begin{remark}[Transient effects]\label{remark:constant_refs}
 		The value of $\mathcal{H}_{s}$ in \eqref{eq:tracking_index1} is heavily influenced by closed-loop transients. Nonetheless, this bias in the evaluation of Fair-MPC can be removed by replacing $\mathcal{H}_{s}$ in \eqref{eq:tracking_index1} with either
 		\begin{equation}\label{eq:tracking_index2}
 			\mathcal{H}_{s}=e^{-\frac{1}{N}\sum_{i=1}^{N}\|\mathbf{x}_{s}^{i}-\mathbf{x}_{T-1}^{i}\|_{2}},
 		\end{equation}    
 		focused on evaluating tracking performance at $T-1$, or
    	\begin{equation}\label{eq:tracking_index3}
 		\mathcal{H}_{s}=\frac{1}{T-\tau_{s}-1}\sum_{t=\tau_{s}}^{T-1}e^{-\frac{1}{N}\sum_{i=1}^{N}\|\mathbf{x}_{s}^{i}-\mathbf{x}_{t}^{i}\|_{2}},
 	\end{equation}    
        that results in an evaluation of performance starting from a user-defined instant $\tau_{s}$. When considering \eqref{eq:tracking_index2}, $\mathcal{H}_{s}$ is equal to $1$ if and only if all systems have achieved their objectives at the last time instant Fair-MPC has been deployed. Instead, with \eqref{eq:tracking_index2} $\mathcal{H}_{s}=1$ is attained when the individual objectives have been attained after $\tau_{s}$ steps.
 	\end{remark}
 	\subsubsection{Assessing equality}
 	The \emph{Jain} index \cite{jain1998} is a quantitative indicator of equality, often used when solving bandwidth allocation problems (see \emph{e.g., \cite{Attiah2018}}). Here we exploit this index to \emph{evaluate} the quality of Fair-MPC, by translating it to our context as follows:   
 	\begin{equation}\label{eq:Jain_index}
 		\mathcal{J}(\mathbf{u}_{t})=\frac{\left(\sum_{i=1}^{N}\|\mathbf{u}_{t}^{i}\|_{1}\right)^{2}}{N \cdot \left(\sum_{i=1}^{N} \left(\|\mathbf{u}_{t}^{i}\|_{1}\right)^{2}\right)} \in \left[\frac{1}{N},1\right],
 	\end{equation}   	
    toward accounting for the instantaneous equality of the control action. Note that, the minimum value of this index is attained when all systems but one evolve in open-loop (\emph{i.e.,} only one system benefits from the available resources). Meanwhile, $\mathcal{J}(\mathbf{u}_{t})$ becomes equal to 1 when the control actions of all systems are perfectly equal. 
    \begin{remark}[Jain index and equality loss]
        In designing the equality loss in \eqref{eq:equality_cost}, we aimed at penalizing more situations in which resources are not equally distributed, which would lead to lower Jain indexes, while promoting scenarios in which $\mathcal{J}(\mathbf{u}_{t})$ approaches 1 by distributing $\bar{U}_{t}$ equally and leveraging all available control resources. Therefore, the rationale behind $J_{u}(\mathbf{u},\bar{U})$ is aligned with that governing the Jain index, bridging between the performance index and the equality loss.
    \end{remark}

   Although allowing us to \textquotedblleft measure\textquotedblright \ fairness, the Jain index in \eqref{eq:Jain_index} is not yet comparable with those defined in \eqref{eq:tracking_index1}-\eqref{eq:tracking_time}. To overcome this issue, $\mathcal{J}(\mathbf{u}_{t})$ is scaled to lay in $[0,1]$, namely
 	\begin{equation}\label{eq:equality_indext}
 		\bar{\mathcal{J}}(\mathbf{u}_{t})=\frac{N \mathcal{J}(\mathbf{u}_{t})-1}{N-1},
 	\end{equation} 
 	and averaged over the $T$ steps Fair-MPC is applied. The \emph{equality} index  $\mathcal{H}_{u} \in [0,1]$ is therefore given by:
 	\begin{equation}\label{eq:equality_index}
 		\mathcal{H}_{u}=\frac{1}{T}\sum_{t=0}^{T-1} \bar{\mathcal{J}}(\mathbf{u}_{t}).
 	\end{equation}
 	We stress that, by definition, this performance indicator is referred to either the whole set of systems or to a subgroup, while it cannot be computed for a single system\footnote{Since $\mathcal{H}_{u}$ quantitatively assesses a group property, its computation for a single system would not provide any relevant insight on the fairness of the control policy.}.
 	\subsubsection{Checking equity}
 	To quantify the equity of the Fair-MPC strategy we directly exploit the same reasoning used to construct $J_{e}(\mathbf{\tilde{x}},\mathbf{x}_{s})$ in \eqref{eq:equity_cost}, namely we compare the individual errors $\mathbf{e}_{t}^{i}=\mathbf{x}_{s}^{i}-\mathbf{x}_{t}^{i}$, for $i=1,\ldots,N$. Specifically, let $\mathcal{E}(\mathbf{x}_{t},\mathbf{x}_{s})$ indicate the average deviation of the individual tracking errors with respect to the error mean, \emph{i.e.,}
 	\begin{equation}
 		\mathcal{E}(\mathbf{x}_{t},\mathbf{x}_{s})=\frac{1}{N}\sum_{i=1}^{N}\bigg\|\mathbf{e}_{t}^{i}-\frac{1}{N}\sum_{j=1}^{N} \mathbf{e}_{t}^{j} \bigg\|_{2}.
 	\end{equation} 
 	The instantaneous index of equity is accordingly defined as
 	\begin{equation}\label{eq:equity_indext}
 		\mathcal{J}_{e}(\mathbf{e}_{t})=e^{-\mathcal{E}(\mathbf{y}_{t},\mathbf{r}_{t})},
 	\end{equation}
 	whose value is equal to $1$ when Fair-MPC results into perfect equity, \emph{i.e.,} all the systems comparably attain their equilibrium states, while it decreases to value close to $0$ whenever there are strong disparities in the systems' state trajectories with respect to their individual goals. 
 	This index is then averaged over $T$ to obtain the ultimate \emph{equity index} $\mathcal{H}_{e} \in [0,1]$, given by:
 	\begin{equation}\label{eq:equity_index}
 		\mathcal{H}_{e}=\frac{1}{T} \sum_{t=0}^{T-1} \mathcal{J}(\mathbf{e}_{t}).
 	\end{equation}
 	As for the equality index in \eqref{eq:equality_index}, $\mathcal{H}_{e}$ loses meaning when computed for a single system, being linked to a group property.
	\subsection{Weight tuning in Fair-MPC}\label{subsec:weights}
	The scaling factors $\gamma_{u}$ and $\mathbf{\Gamma}_{e}$ in \eqref{eq:scaled_weights} are problem specific, \emph{i.e.,} they depend on the number of systems, the available control resources and the individual targets, but they are not linked to the relative importance of fairness with respect to individual performance. On the other hand, the scaled weights $\bar{\rho}^{i}$ and $\mathbf{\bar{W}}^{i}$ are directly associated with the significance of fairness in the design of the control action. As such, $\gamma_{u}$ and $\mathbf{\Gamma}_{e}$ can eventually be prefixed or computed at $t=0$, by looking at the difference in magnitude between the three terms in the cost when applying the first Fair-MPC action obtained with unitary weights and then they can be kept constant. Instead, $\bar{\rho}^{i}$ and $\mathbf{\bar{W}}^{i}$ should be carefully calibrated and (possibly) adapted.
	
	A possible strategy to adaptively tune these hyper-parameters while accounting for the fairness of the current control action can be devised by exploiting the performance indexes introduced in Section~\ref{Subsec:performance_indexes}. Indeed, by monitoring the evolution of the fairness criteria $\bar{\mathcal{J}}(\mathbf{u}_{t})$ and $\mathcal{J}_{e}(\mathbf{e}_{t})$ in \eqref{eq:equality_indext} and \eqref{eq:equity_indext}, respectively, the relative importance weights $\bar{\rho}_{t}^{i}$ and $\mathbf{\bar{W}}_{t}^{i}$ can be increased/decreased whenever one or both fairness principles are poorly/satisfactorily attained. Accordingly, these hyper-parameters can be selected to be inversely proportional to the associated performance index. At the same time, enforcing an even distribution of control resources can be counterproductive when systems progressively complete their tasks, as it naturally prevent those systems still far from their targets to receive the resources needed to attain their goals. 
 
 According to this reasoning, the weights can be defined as:
    \begin{subequations}\label{eq:varying_weights}
		\begin{align}
			& \bar{\rho}_{t}^{i}=\bar{\rho}_{t}=\begin{cases}
				\frac{1}{\bar{\mathcal{J}}(\mathbf{u}_{t})},~~\mbox{if } t \leq \bar{t},\\
				\frac{\bar{\rho}_{t-1}^{i}}{2}, \quad~~ \mbox{otherwise},	
			\end{cases}\label{eq:varying_weights1}\\
			& \mathbf{\bar{W}}_{t}^{i}=\mathbf{\bar{W}}_{t}=\frac{1}{\mathcal{J}_{e}(\mathbf{e}_{t})},\label{eq:varying_weights2}
		\end{align}
	\end{subequations}
	where $\bar{t}$ is a tunable parameter that can be selected based on the percentage of systems whose initial state lay within a prefixed error band with respect to their targets, or as a prefixed number of time steps. Note that this choice promotes penalties that generally weigh more poor performance in terms of equity rather than equality, due to the negative effect this last fairness criteria can have on individual performance in the long run.	
\section{Numerical examples}\label{sec:examples}
We now assess the effect of introducing fairness in a predictive control design routine through two numerical examples. By focusing on a simple two-system case, we discuss the effect induced on both individual performance and resource allocation by the progressive penalization of unfair control actions according to the proposed framework. On the same example, we further analyze the benefits of exploiting the fairness-oriented auto-tuning strategy proposed in Section~\ref{subsec:weights}. Fair-MPC is then applied to control the motion of a set of systems on a bi-dimensional plane. We initially consider again a simple two-system scenario, within which we specifically focus on assessing the effect that exhaustible resources have on closed-loop performance by leveraging \eqref{eq:input_2} (but not modifying the allocation constraint \eqref{eq:FMPC_alloconstr}). We then shift to a more challenging two-class case and $N=8$, to assess the closed-loop impact of the proposed penalties on equity and equality when reshaped as in \eqref{eq:class_based}. In all our examples $\mathbf{P}$ in \eqref{eq:terminal_cost} is set to $I_{nN}$, thus leaving the tunable parameter $\lambda$ as our degree of freedom to calibrate the terminal cost.

\subsection{A two-system example}

Consider $N=2$ first order systems, respectively characterized by the following dynamics:
\begin{equation}\label{eq:two_systems}
\mathcal{S}^{1}\!\!:~x_{t+1}^{1}=0.4x_{t}^{1}+0.1u_{t}^{1},~~~\mathcal{S}^{2}\!\!:~x_{t+1}^{2}=0.9x_{t}^{2}+0.1u_{t}^{2}.
\end{equation}   
Assume that they both share the same initial conditions $x_{0}^{1}=x_{0}^{2}=0$, while aiming at attaining the same equilibrium point $x_{s}^{1}=x_{s}^{2}=2$ within $T=20$ steps. To attain this goal, the parameters of Fair-MPC are selected as in \tablename{~\ref{tab:2systems_param}}, where the weights are equal for both systems and the upper-bound in \eqref{eq:shared_res} is assumed to be constant at all times. Note that, the chosen $\bar{U}$ does not allow the systems to attain their target by design (\emph{i.e.,} the chosen equilibrium is not a feasible point). Despite we know beforehand that the systems cannot attain their targets, this analysis allows us to clearly emphasize the effects that incorporating fairness in control has on closed-loop performance. In this scenario, we thus progressively introduce our fairness-oriented penalties into the control design problem, to gradually assessing their impact on the closed-loop.

\begin{table}[!hp]
	\caption{Two-system example: parameters of the Fair-MPC problem.}\label{tab:2systems_param}
	\centering
	\begin{tabular}{cccccccccc}
		  $L$ & $\bar{U}$ & $Q$ & $\bar{\rho}$ & $\gamma_{u}$ & $\bar{W}$ & $\Gamma_{e}$ & $\beta$ & $\lambda_x$ & $\lambda_u$ \\
		\hline 
		\hline
		  20 & 10 & 1 & 3 & $10^{-1}$ &  1 & $10$ & 0.1 & 0.1 & 0.1\\
		\hline
	\end{tabular}
\end{table}
\begin{figure}[!hp]
	\centering
	\begin{tabular}{c}
	\subfigure[Performance only: $J_{L}^{\mathrm{fair}}(\mathbf{\tilde{u}},\mathbf{\tilde{x}},\bar{U},\mathbf{x}_s)=J_{p}(\mathbf{\tilde{x}},\mathbf{x}_{s})$\label{fig:2sys_adding1}]{\begin{tabular}{cc}
		\includegraphics[width=0.40\columnwidth]{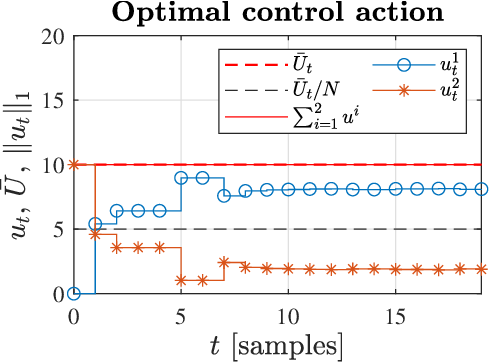} &
		\includegraphics[width=0.40\columnwidth]{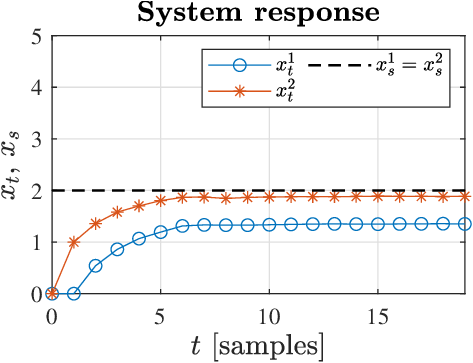}
	\end{tabular}}\\
	\subfigure[Performance+equality: $J_{L}^{\mathrm{fair}}(\mathbf{\tilde{u}},\mathbf{\tilde{x}},\bar{U},\mathbf{x}_s)=J_{p}(\mathbf{\tilde{x}},\mathbf{x}_{s})+J_{u}(\tilde{u},\bar{U})$\label{fig:2sys_adding2}]{\begin{tabular}{cc}
			\includegraphics[width=0.40\columnwidth]{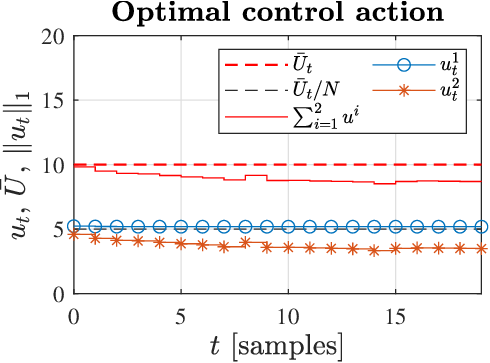} &
			\includegraphics[width=0.40\columnwidth]{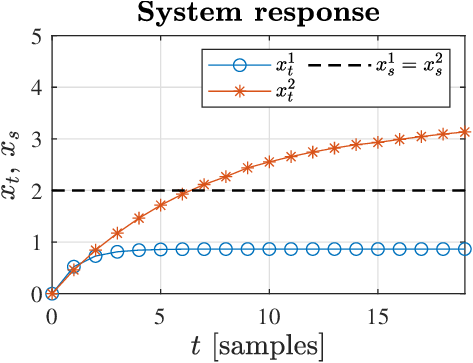}
	\end{tabular}}\\
	\subfigure[Performance+equity: $J_{L}^{\mathrm{fair}}(\mathbf{\tilde{u}},\mathbf{\tilde{x}},\bar{U},\mathbf{x}_s)=J_{p}(\mathbf{\tilde{x}},\mathbf{x}_{s})+J_{e}(\mathbf{\tilde{x}},\mathbf{x}_{s})$\label{fig:2sys_adding3}]{\begin{tabular}{cc}
			\includegraphics[width=0.40\columnwidth]{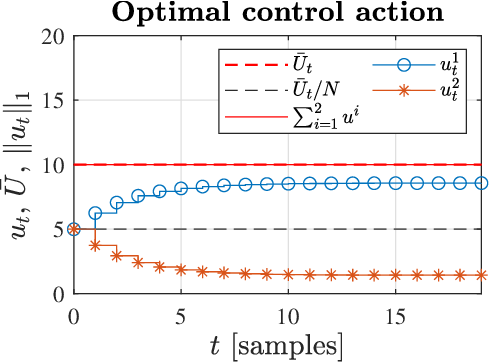} &
			\includegraphics[width=0.40\columnwidth]{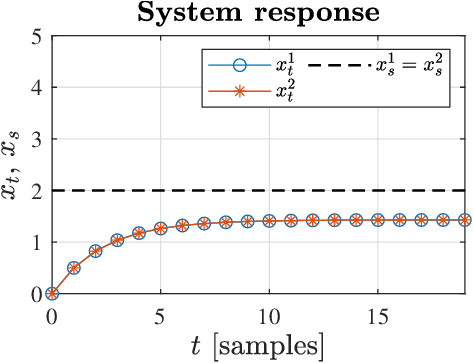}
	\end{tabular}}\\
	\subfigure[Fair-MPC: $J_{L}^{\mathrm{fair}}(\mathbf{\tilde{u}},\mathbf{\tilde{x}},\bar{U},\mathbf{x}_s)$ in \eqref{eq:cost_divided}\label{fig:2sys_adding4}]{\begin{tabular}{cc}
			\includegraphics[width=0.40\columnwidth]{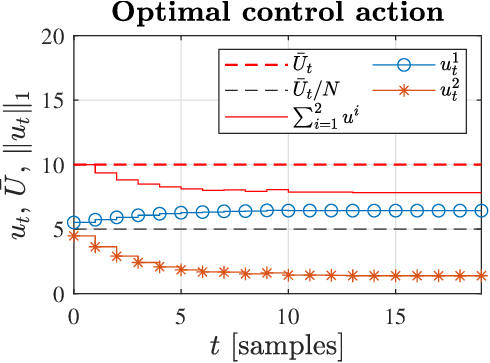} &
			\includegraphics[width=0.40\columnwidth]{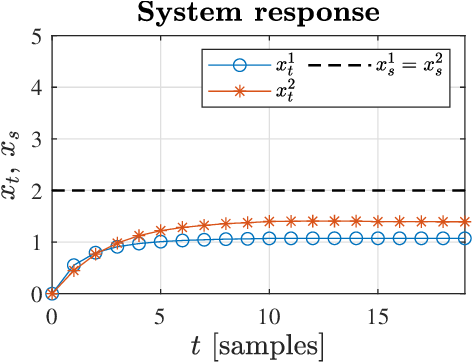}
	\end{tabular}}
	\end{tabular}
	\caption{Two-system example: control actions and systems' states \emph{vs} control strategy. In all cases weights are constant and pre-fixed.}
	\label{fig:2sys_adding}
\end{figure}
\begin{table}[!hp]
	\caption{Two-system example: performance indexes \emph{vs} control strategy.}\label{Tab:2ag_perf}
	\centering
	\begin{tabular}{|c|c|c|c|}
		\cline{2-4}
		\multicolumn{1}{c|}{} & $\mathcal{H}_{s}$ \eqref{eq:tracking_index2} & $\mathcal{H}_{u}$ \eqref{eq:equality_index} & $\mathcal{H}_{e}$ \eqref{eq:equity_index}\\
		\hline 
		Performance only &\textbf{0.683} & 0.497 & 0.744 \\
		\hline
		Performance+equality &0.316 & \textbf{0.945} & 0.814\\
		\hline
		Performance+equity & 0.565 & 0.441 & \textbf{0.999}\\
		\hline 
		Fair-MPC & 0.464 & 0.529 & 0.876 \\
		\hline
	\end{tabular}
\end{table}

\paragraph{Performance only} Let us initially set $\rho$ and $W$ to zero for both systems to solely look at the tracking performance to design the control action. As shown in \figurename{~\ref{fig:2sys_adding1}}, the constraint in \eqref{eq:shared_res} is always active even if both systems do not attain their targets. Meanwhile, the initial control effort of $\mathcal{S}_{2}$ is greater than the one of $\mathcal{S}_{1}$, consequently allowing the second system to approach the target faster than the other. Most of control resources are later on invested into the first system, which still does not attain the target. Clearly, a strategy focused on tracking only leads to huge disparities between the systems, resulting in an unfair controller, as further highlighted by the performance indexes reported in \tablename{~\ref{Tab:2ag_perf}}. 
\paragraph{Performance+equality} By setting $\rho$ as in \tablename{~\ref{tab:2systems_param}}, while imposing $W=0$, we now assess the effect of equality on the systems' performance. As visible in \figurename{~\ref{fig:2sys_adding2}}, penalizing inequalities in addition to individual performance results in control actions that lean towards the average $\frac{\bar{U}}{2}$, to provide all systems with equal opportunities. As a consequence, on the one hand, $\mathcal{S}_{1}$ is fed with an input that does not allow it to achieve its target, while the control resources are not fully exploited throughout the considered time span. On the other hand, the input of $\mathcal{S}_{2}$ makes it exceed the chosen equilibrium point, while slowing down its initial convergence. As confirmed by the equity index $\mathcal{H}_{e}$ reported in \tablename{~\ref{Tab:2ag_perf}}, a \emph{blind} egalitarian allocation of resources eventually exacerbates differences among the systems, in this case increasing the disparities between their responses. 
\paragraph{Performance+equity}     
By imposing $\rho=0$ and keeping $W$ as in \tablename{~\ref{tab:2systems_param}}, we then study the effect of equity on closed-loop performance. From the results reported in \figurename{~\ref{fig:2sys_adding3}} it is clear that the introduction of equity mitigates performance disparities among the systems, at the price of a slight deterioration in the tracking performance of $\mathcal{S}_{2}$. Meanwhile, this strategy creates inequalities in the allocation of control resources, as it tends to support under-performing systems over over-performing ones. These considerations are supported by the tracking and equality indexes shown in \tablename{~\ref{Tab:2ag_perf}}, highlighting a slight deterioration in individual performance (see $\mathcal{H}_{e}$). 
\paragraph{Fair-MPC} We now combine the effects of equality and equity, by designing the Fair-MPC law with the parameters reported in \tablename{~\ref{tab:2systems_param}}. As shown in both \figurename{~\ref{fig:2sys_adding4}} and \tablename{~\ref{tab:2systems_param}}, the inclusion of both fairness principles in control design allows us to obtain a slightly more even distribution of control resources with respect to the purely tracking case, while balancing the improvement in the performance of $\mathcal{S}_{1}$ and the slight deterioration of the ones of $\mathcal{S}_{2}$ characterizing the strategy solely looking at performance and equity.  
\paragraph{An outlook on closed-loop disparities} 
\begin{table}[!hp]
	\caption{Two-system example: tracking index in \eqref{eq:individual_index} \emph{vs}  strategy.}\label{Tab:2agen_single}
	\centering
	\begin{tabular}{|c|c|c|}
		\cline{2-3}
		\multicolumn{1}{c|}{} & $\mathcal{H}_{s}^{1}$ & $\mathcal{H}_{s}^{2}$\\
		\hline
		Performance only & 0.523 & 0.893\\
		\hline
		Individual performance+equality & 0.321 & 0.310\\
		\hline
		Individual performance+equity & 0.565 & 0.565\\
		\hline
		Fair-MPC & 0.395 & 0.545\\
		\hline
	\end{tabular}
\end{table}
By additionally considering the individual indexes
\begin{equation}\label{eq:individual_index}
	\mathcal{H}_{s}^{i}=e^{-\|\mathbf{x}_{s}^{i}-\mathbf{x}_{T-1}^{i}\|_{2}},
\end{equation}       
reported in \tablename{~\ref{Tab:2agen_single}}, all the results related to individual closed-loop performance highlighted in the previous discussions are even more evident. Indeed, while looking at individual performance in control design maintains the \textquotedblleft natural\textquotedblright \ disparity between systems, introducing equity demolishes such difference at the price of a drop in performance of $\mathcal{S}_{2}$. Meanwhile introducing equality instead of equity results in a control action that preserves differences between systems, while causing a general degradation of their tracking performance. This drop is experienced also when deploying the Fair-MPC scheme, but (as imposed) discrepancies between agents are reduced.   
\subsubsection{Fair-MPC with auto-tuning}

Let us now increase the overall resources available at each time instant to $\bar{U}=20$ to focus on a case where the systems can actually attain their targets (see \figurename{~\ref{fig:2sys_FAIR1}}). In this scenario, excessively encouraging an equal distribution of the control resources can be counterproductive, since it might lead the two systems to poorly attain their targets or/and increase the disparity between them. Meanwhile, prompting equity might only result in a waste of resources and their unfair distribution. We thus evaluate the benefits of adapting the weights of Fair-MPC as proposed in Section~\ref{subsec:weights}, when compared with using prefixed penalties. To this end, we fix all the parameters to the ones in \tablename{~\ref{tab:2systems_param}} except for $\bar{U}$, $\bar{\rho}$ and $\bar{W}$ and we consider two possible variants of \eqref{eq:varying_weights1}, subsequently denoted as \emph{case A} and \emph{case B}. While in both cases we select $\bar{t}$ as the instant at which the following relationship is verified:
\begin{equation}\label{eq:t_bar}
	x_{s}^{i}-x_{t}^{i}<0,~~~~\bar{t}-0.2T \leq t \leq \bar{t}
\end{equation} 
for at least one of the systems, thus reducing the relative importance of equality when a system is consistently above its target, in \emph{case B} we also neglect the halving of $\bar{\rho}_{t}^{i}$ in \eqref{eq:varying_weights1}, for $i=1,2$. As it can be clearly deduced from the results shown in \figurename{~\ref{fig:2sys_FAIR2}}-\ref{fig:2sys_FAIR3} and the indexes reported in \tablename{~\ref{tab:2ag_halving}}, the introduction of time-varying weights allows us to balance the counteract long-term effects of the equality-inducing penalty \eqref{eq:equality_cost}, which naturally tends to steer the inputs toward the same value (hence deteriorating the individual tracking performance). At the same time, the weights considered in \emph{case A} allow Fair-MPC to progressively phase out the initial overshoot with respect to the reference characterizing $\mathcal{S}_{2}$. Instead, the penalties considered in \emph{case B} do not account for the redundancy of the available resources, leading to individual control actions that are excessive with respect to the actual needs of each system once again due to the effect of the equality-oriented term in \eqref{eq:cost_divided}.       
\begin{figure}[!hp]
	\centering
	\begin{tabular}{c}
		\subfigure[Performance only with $\bar{U}=20$ and prefixed weights \label{fig:2sys_FAIR1}]{
			\begin{tabular}{cc}
				\includegraphics[width=0.40\columnwidth]{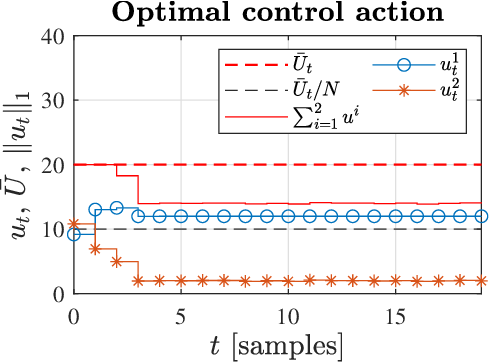} 
				&
				\includegraphics[width=0.40\columnwidth]{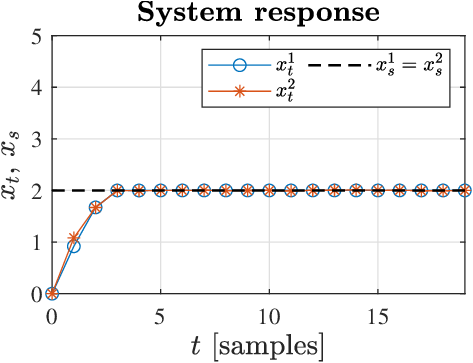}
		\end{tabular}}\\
		\subfigure[Fair-MPC with $\bar{U}=20$ and prefixed weights\label{fig:2sys_FAIR2}]{
			\begin{tabular}{cc}
				\includegraphics[width=0.40\columnwidth]{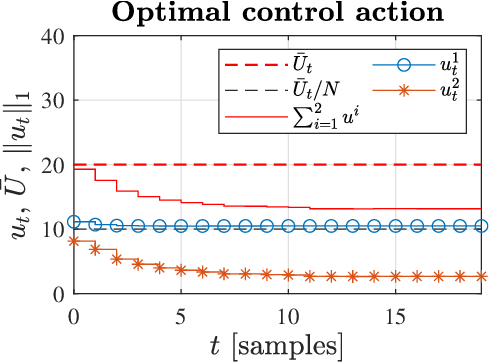} 
				&
				\includegraphics[width=0.40\columnwidth]{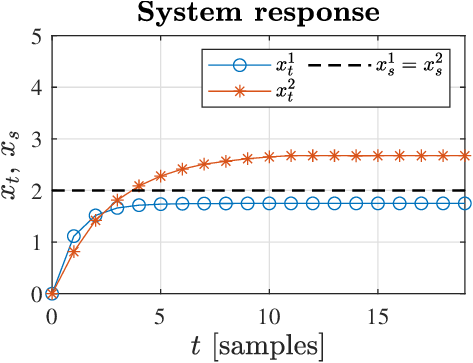}
		\end{tabular}}\\
		\subfigure[Fair-MPC with $\bar{U}=20$ and auto-tuned weights\label{fig:2sys_FAIR3}]{
			\begin{tabular}{cc}
			\includegraphics[width=0.40\columnwidth]{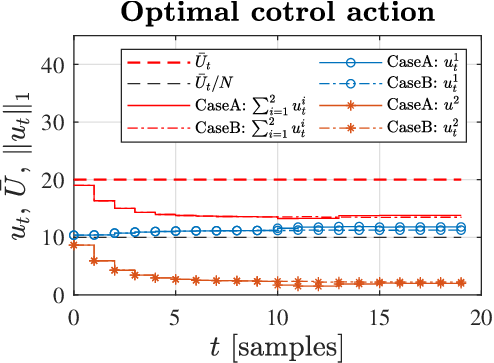} &
			\includegraphics[width=0.40\columnwidth]{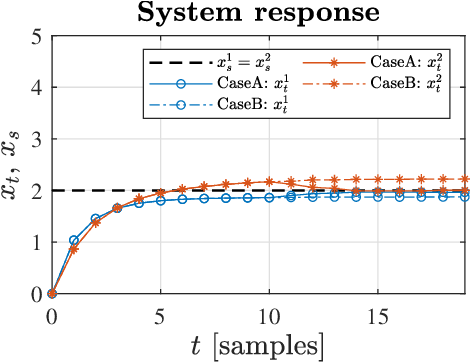}
		\end{tabular}}
	\end{tabular}
	\caption{Two-system example: control actions and systems' states \emph{vs} tuning strategy.}
	\label{fig:2sys_FAIR}
\end{figure}
\begin{table}[!hp]
	\caption{Two-system example: performance \emph{vs} tuning strategy.}	\label{tab:2ag_halving}
	\centering
	\begin{tabular}{|c|c|c|c|}
		\cline{2-4}
		\multicolumn{1}{c|}{} & \hspace*{-.1cm}$\mathcal{H}_{s}$\hspace*{-.1cm} \eqref{eq:tracking_index2} \hspace*{-.1cm}&\hspace*{-.1cm} $\mathcal{H}_{u}$ \eqref{eq:equality_index} \hspace*{-.1cm}&\hspace*{-.1cm} $\mathcal{H}_{e}$ \eqref{eq:equity_index}\\
		\hline
		Performance only+fixed weights & 0.999 & 0.399 & 0.994 \\
            \hline
		Fair-MPC+fixed weights & 0.630 & 0.584 & 0.851 \\
		\hline
		Fair-MPC+auto-tuned weights (Case A) & \textbf{0.980} & \textbf{0.441} & \textbf{0.975}\\
		\hline
		Fair-MPC+auto-tuned weights (Case B) & 0.840 & 0.480 & 0.960\\
		\hline
	\end{tabular}
\end{table}
\subsubsection{Fair-MPC on unstable open-loop systems} 

Both the systems considered so far are open-loop stable (see \eqref{eq:two_systems}). We thus shift to two new systems, \emph{i.e.,}
\begin{equation}\label{eq:two_systems2}
\mathcal{S}^{1}\!\!:~x_{t+1}^{1}=0.5x_{t}^{1}+0.1u_{t}^{1},~~~\mathcal{S}^{2}\!\!:~x_{t+1}^{2}=1.5x_{t}^{2}+0.1u_{t}^{2},
\end{equation}   
so that one of the two ($\mathcal{S}^{2}$) is open-loop unstable. 
To stabilize both systems to the origin, we now evaluate the impact of progressively introducing fairness in control design, with an eye on checking closed-loop stability as well as the attained performance. Note that, in this new scenario, the hyper-parameters of Fair-MPC are set as in \tablename{~\ref{tab:2systems_param2}}.

\figurename{~\ref{fig:2sys_adding_2}} shows that the states and inputs of both the considered systems converge after the transient, thus showing that the proposed scheme allows us to stabilize the system despite the introduction of our fairness-driven terms in the cost of \eqref{eq:F-MPC}. At the same time, these results highlight that zero-state is reached both when tracking only is penalized and when equity is introduced in the design cost. Nonetheless, in this second case the transient behavior of $\mathcal{S}^{2}$ is slightly modified, with the designed input visibly acting to reduce the differences between the systems' states. Meanwhile, it is clear that incorporating only equality reduces discrepancies between the individual optimal inputs during the transient, at the price of making the systems' states not converge to zero. While not exactly allowing for convergence to zero either, Fair-MPC achieves the desired trade-off between individual performance, equality and equity (see, \emph{e.g.,} the inputs and states during the transient). These results are further confirmed by the indexes in \tablename{~\ref{Tab:2ag_perf2}}.   

\begin{table}[!hp]
	\caption{Two-system example (with open-loop unstable): parameters of the Fair-MPC problem.}\label{tab:2systems_param2}
	\centering
	\begin{tabular}{cccccccccc}
		  $L$ & $\bar{U}$ & $Q$ & $\bar{\rho}$ & $\gamma_{u}$ & $\bar{W}$ & $\Gamma_{e}$& $\beta$ &$\lambda_x$&$\lambda_u$\\
		\hline 
		\hline
		 20 & 10 & 1 & 1 & $10^{-2}$ &  1 & $1$ & 0.1 & 0.1 & 0.1\\
		\hline
	\end{tabular}
\end{table}
\begin{figure}[!hp]
	\centering
	\begin{tabular}{c}
	\subfigure[Performance only: $J_{L}^{\mathrm{fair}}(\mathbf{\tilde{u}},\mathbf{\tilde{x}},\bar{U},\mathbf{x}_s)=J_{p}(\mathbf{\tilde{x}},\mathbf{x}_{s})$\label{fig:2sys_stability1}]{\begin{tabular}{cc}
		\includegraphics[width=0.40\columnwidth]{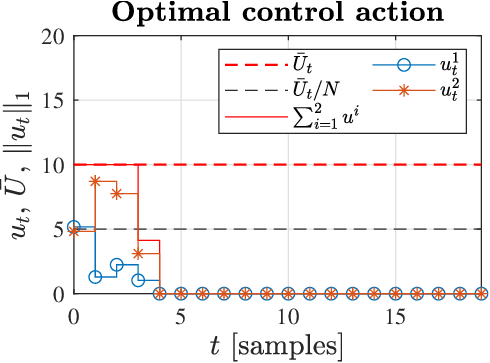} &
		\includegraphics[width=0.40\columnwidth]{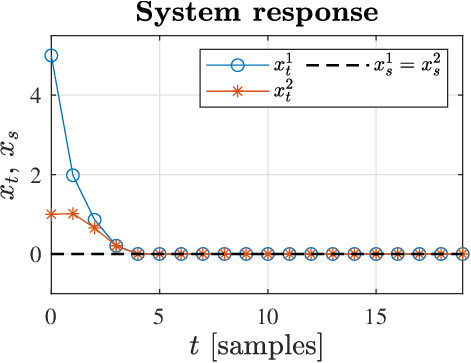}
	\end{tabular}}\\
	\subfigure[Performance+equality: $J_{L}^{\mathrm{fair}}(\mathbf{\tilde{u}},\mathbf{\tilde{x}},\bar{U},\mathbf{x}_s)=J_{p}(\mathbf{\tilde{x}},\mathbf{x}_{s})+J_{u}(\tilde{u},\bar{U})$\label{fig:2sys_stability2}]{\begin{tabular}{cc}
			\includegraphics[width=0.40\columnwidth]{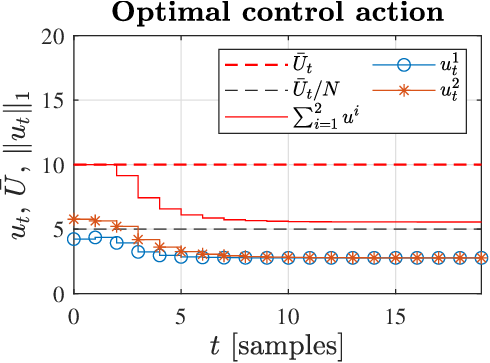} &
			\includegraphics[width=0.40\columnwidth]{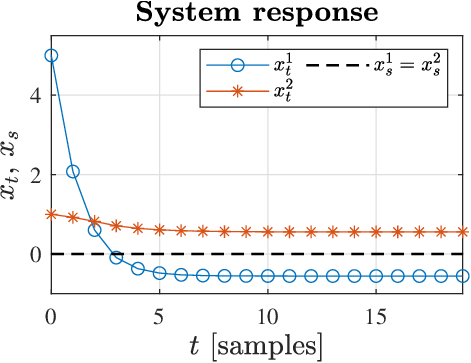}
	\end{tabular}}\\
	\subfigure[Performance+equity: $J_{L}^{\mathrm{fair}}(\mathbf{\tilde{u}},\mathbf{\tilde{x}},\bar{U},\mathbf{x}_s)=J_{p}(\mathbf{\tilde{x}},\mathbf{x}_{s})+J_{e}(\mathbf{\tilde{x}},\mathbf{x}_{s})$\label{fig:2sys_stability3}]{\begin{tabular}{cc}
			\includegraphics[width=0.40\columnwidth]{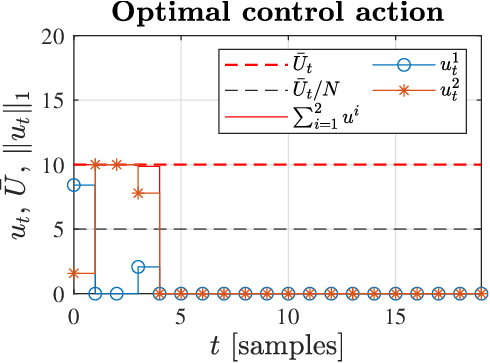} &
			\includegraphics[width=0.40\columnwidth]{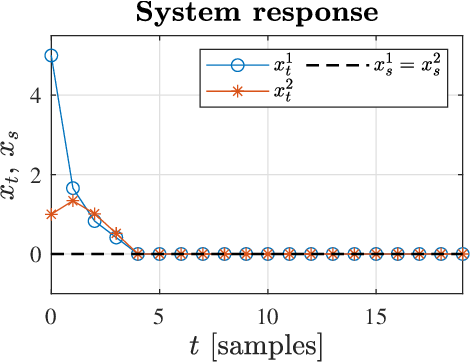}
	\end{tabular}}\\
	\subfigure[Fair-MPC: $J_{L}^{\mathrm{fair}}(\mathbf{\tilde{u}},\mathbf{\tilde{x}},\bar{U},\mathbf{x}_s)$ in \eqref{eq:cost_divided}\label{fig:2sys_stability4}]{\begin{tabular}{cc}
			\includegraphics[width=0.40\columnwidth]{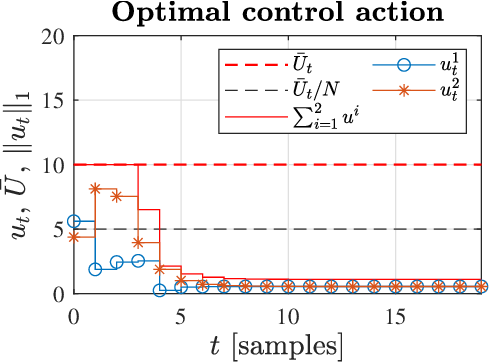} &
			\includegraphics[width=0.40\columnwidth]{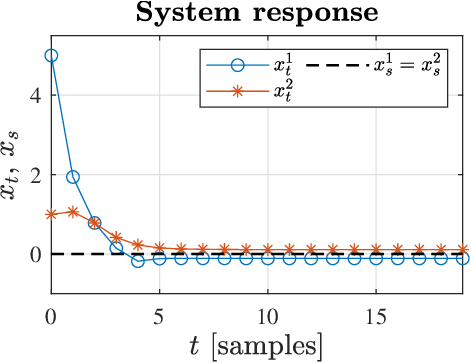}
	\end{tabular}}
	\end{tabular}
	\caption{Two-system example (with open-loop unstable): control actions and systems' states \emph{vs} control strategy.}
	\label{fig:2sys_adding_2}
\end{figure}
\begin{table}[!hp]
	\caption{Two-system example (with open-loop unstable): performance \emph{vs} strategy.}\label{Tab:2ag_perf2}
	\centering
	\begin{tabular}{|c|c|c|c|}
		\cline{2-4}
		\multicolumn{1}{c|}{} & $\mathcal{H}_{s}$ \eqref{eq:tracking_index2} & $\mathcal{H}_{u}$ \eqref{eq:equality_index} & $\mathcal{H}_{e}$ \eqref{eq:equity_index}\\
		\hline 
		Performance only & \textbf{1.000} & 0.605 & 0.971 \\
		\hline
		Performance+equality &0574 & \textbf{0.990} & 0.939\\
		\hline
		Performance+equity & 0.999 & 0.820 & \textbf{0.985}\\
		\hline 
		Fair-MPC & 0.895 & 0.896 & 0.969 \\
		\hline
	\end{tabular}
\end{table}
\subsection{Motion on a bi-dimensional space}
Consider now a set of systems moving on a bi-dimensional space, whose individual behavior is described by the following dynamics:
\begin{subequations}\label{eq:motion}
\begin{align}
	&\mathbf{x}_{t+1}^{i}=\begin{bmatrix}
		1 & 0 & 1 & 0\\
		0 & 1 & 0 & 1\\
		0 & 0 & 1 & 0\\
		0 & 0 & 0 & 1
	\end{bmatrix}\mathbf{x}_{t}^{i}+\begin{bmatrix}
	0 & 0\\
	0 & 0\\
	[\mathbf{b}^{i}]_{1} & 0\\
	0 & [\mathbf{b}^{i}]_{2}
	\end{bmatrix}\mathbf{u}_{t}^{i},
\end{align}
\end{subequations}
where the first two states indicate the position of each system over time, while the remaining ones are their velocities, and $[\mathbf{b}^{i}]_{j}$ characterize the peculiar effect that the $j$-th component of the control action has on the evolution of the $i$-th system. Note that higher values of these parameters imply that the performance of the corresponding system(s) is considerably shaped by the external inputs. Accordingly, in what follows we denote the systems with higher $[\mathbf{b}^{i}]_{j}$ as \emph{influenced}, while the other are called \emph{refrained}. Moreover, for the sake of the readability of subsequent results, we introduce the vector containing only the positions of the analyzed systems, namely
\begin{equation}\label{eq:position_def}
\mathbf{p}_{t}^{i}=\begin{bmatrix}
    [\mathbf{x}_{t}^{i}]_{1} & [\mathbf{x}_{t}^{i}]_{2}
\end{bmatrix}^{\top}.
\end{equation}

In our analysis, we assume that the systems aim at moving towards their individual target equilibrium positions $\mathbf{p}_{s}^{i}\in \mathbb{R}^{2}$ from the origin of the 2-dimensional plane, finally stopping on the equilibrium position itself. Therefore, the equilibrium state is always given by  
\begin{equation*}
\mathbf{x}_{s}^{i}=\begin{bmatrix}(\mathbf{p}_{s}^{i})^{\top} & 0 & 0\end{bmatrix},~~~i=1,\ldots,N.
\end{equation*}
The performance of Fair-MPC is assessed over a simulation span of $T=20$ steps, while automatically tuning the weights as in \eqref{eq:varying_weights}, and choosing $\bar{t}$ as in \eqref{eq:t_bar}. Lastly, we only show the 1-norms of states and inputs over time, to allow for an easier visual inspection of the attained closed-loop behavior.

\subsubsection{The two-system case}

Let us initially consider a simpler case, in which only two systems ($N=2$) have to be controlled. As shown in \tablename{~\ref{tab:2d_2systems_data}}, one of two systems is advantaged with respect to the other, given its dynamics and the proximity of the initial position to the target equilibrium.  
\begin{table}[!hp]
\caption{Bi-dimensional motion (two-system case): features and targets.}\label{tab:2d_2systems_data}
\centering
\begin{tabular}{|c|c|c|}
	\cline{2-3}
	\multicolumn{1}{c|}{} & $\mathbf{b}^{i}$ & $\mathbf{p}_{s}^{i}$\\
	\hline
	\emph{refrained} system & $\begin{bmatrix}
		0.2 & 0.2
	\end{bmatrix}$ & $\begin{bmatrix}
	10 & -13
\end{bmatrix}$\\
	\hline
	\emph{influenced} system & $\begin{bmatrix}
		1 & 1
	\end{bmatrix}$ &$\begin{bmatrix}
		-7 & 2
	\end{bmatrix}$\\
\hline
\end{tabular}
\end{table}
\begin{table}[!hp]
	\caption{Bi-dimensional motion (two-system case): parameters of the Fair-MPC problem.}\label{tab:2d_param}
	\centering
	\begin{tabular}{ccccccc}
		$L$ & $Q$ & $\gamma_{u}$ & $\Gamma_{e}$& $\beta$ & $\lambda_x$ & $\lambda_u$\\
		\hline 
		\hline
		10 & 1 & $10^{-1}$ & $10^{1}$ &0.1 & 0.1 & 0.1\\
		\hline
	\end{tabular}
\end{table}
By considering the parameters in \tablename{~\ref{tab:2d_param}}, we now analyze how the different choices for $\bar{U}_{t}$ in \eqref{eq:shared_res} (see the discussion in section~\ref{sec:properties}) shape the obtained closed-loop performance and the associated control actions. 
\subsubsection{Uniformly allocated resources} 
We now focus on the scenario where $\bar{U}_{t}=20$ for all $t \geq 0$, namely control resources are unlimited and they are uniformly allocated over time. As shown in \figurename{~\ref{fig:2d_2systems}}, when penalizing performance only the \emph{influenced} system attains its reference slightly faster than the \emph{refrained} one. Meanwhile, little control resources are generally allocated to the first system because of its closeness to its target. Instead, when Fair-MPC is used, the control effort of each system is more balanced, causing the elongation in the response of the \emph{influenced} system with respect to its target. In turn, this phenomenon slows down the convergence of this system, allowing both systems to achieve their goal at approximately the same time instant and, ultimately, leading to a fairer controller, as confirmed by the performance indexes reported in \tablename{~\ref{tab:2d_2systems_KPI}}. 
\begin{table}[!hp]
	\caption{Bi-dimensional motion in closed-loop (two systems with unlimited resources): performance indexes for $\alpha_{\%}=10\%$ in \eqref{eq:tau} \emph{vs} control strategy.}\label{tab:2d_2systems_KPI}
	\centering
	\begin{tabular}{|c|c|c|c|c|}
		\cline{2-5}
		\multicolumn{1}{c|}{} & $\mathcal{H}_{s}$ \eqref{eq:tracking_index2} & $\mathcal{H}_{\tau}$ \eqref{eq:tracking_time}& $\mathcal{H}_{u}$ \eqref{eq:equality_index} & $\mathcal{H}_{e}$ \eqref{eq:equity_index}\\
		\hline
		Performance only &1.000 & \textbf{0.750} & 0.326 & 0.449\\
		\hline
		Fair-MPC &1.000 & 0.700 & \textbf{0.487} & \textbf{0.646} \\
		\hline
	\end{tabular}
\end{table}
\begin{figure}[!hp]
	\centering
	\begin{tabular}{c}
		\subfigure[1-norm of the systems' inputs and positions]{\begin{tabular}{cc}
			\includegraphics[width=0.40\columnwidth]{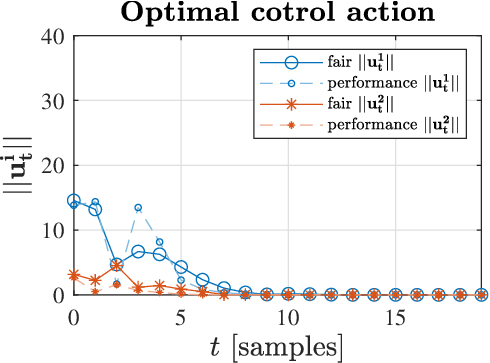} &		\includegraphics[width=0.40\columnwidth]{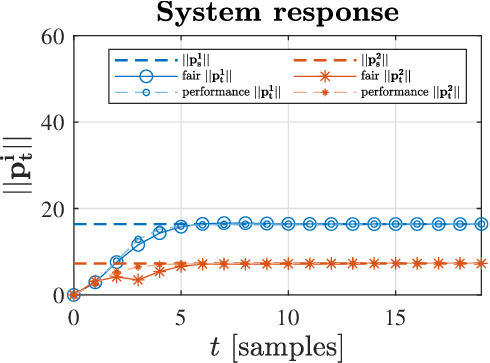} \\
		\end{tabular}}\\
		\subfigure[Motion on the bi-dimensional plane]{\includegraphics[width=0.4\columnwidth]{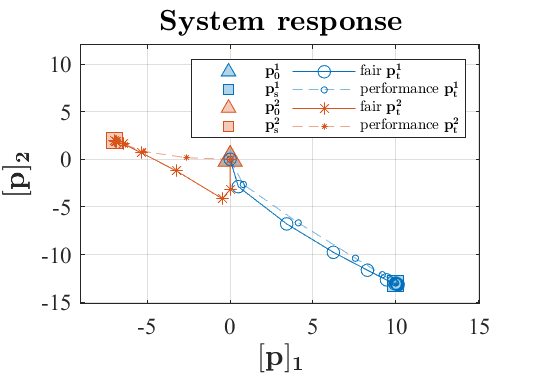}}
	\end{tabular}
	\caption{Bi-dimensional motion in closed-loop (two systems with unlimited resources): Strategy penalizing performance only \emph{vs} Fair-MPC.}\label{fig:2d_2systems}
\end{figure}

\subsubsection{Exhaustible resources} 
Suppose now that the available control resources get reduced depending on their usage over time, starting from\footnote{When the initial budget is lower than $\bar{U}=200$, the tracking policy under-performs Fair-MPC even in terms of individual performance.} $\bar{U}_{0}=200$. As shown in \figurename{~\ref{fig:2d_2systemsEX}}, both control strategies allow the systems to attain their targets. At the same time, more control resources are initially exploited when considering tracking only. This result is particularly relevant in the considered setting, as resources decrease over time and it thus fundamental to preserve them as much as possible. Note that, as expected, having limited resources leads to a change in the trajectories of the two systems with respect to the ones followed in \figurename{~\ref{fig:2d_2systems}}. Meanwhile, as for the case of unlimited resources, the use of Fair-MPC tends to slightly slow down the convergence of the stronger system to its target, for the benefit of an increased fairness among the systems (see the indexes in \tablename{~\ref{tab:2d_2systemsEX_KPI}}). 
\begin{table}[!hp]
	\caption{Bi-dimensional motion in closed-loop (two systems with exhaustible resources): performance indexes with $\alpha_{\%}=10\%$ in \eqref{eq:tau} \emph{vs} control strategy.}\label{tab:2d_2systemsEX_KPI}
	\centering
	\begin{tabular}{|c|c|c|c|c|}
		\cline{2-5}
		\multicolumn{1}{c|}{} & $\mathcal{H}_{s}$ \eqref{eq:tracking_index2} & $\mathcal{H}_{\tau}$ \eqref{eq:tracking_time}& $\mathcal{H}_{u}$ \eqref{eq:equality_index} & $\mathcal{H}_{e}$ \eqref{eq:equity_index}\\
		\hline
		Performance only &1.000 & \textbf{0.825} & 0.176 & 0.630\\
		\hline
		Fair-MPC &1.000 & 0.800 & \textbf{0.365} & \textbf{0.703} \\
		\hline
	\end{tabular}
\end{table}
\begin{figure}[!hp]
	\centering
	\begin{tabular}{c}
		\subfigure[1-norm of the systems' inputs and outputs]{\begin{tabular}{cc}
			\includegraphics[width=0.40\columnwidth]{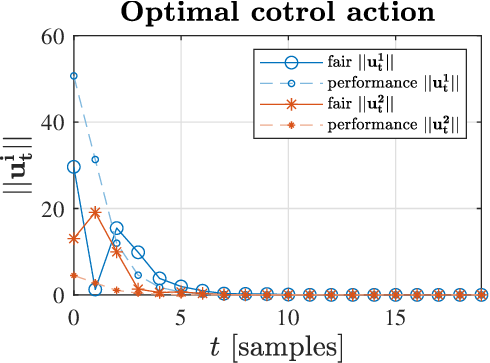} &	\includegraphics[width=0.40\columnwidth]{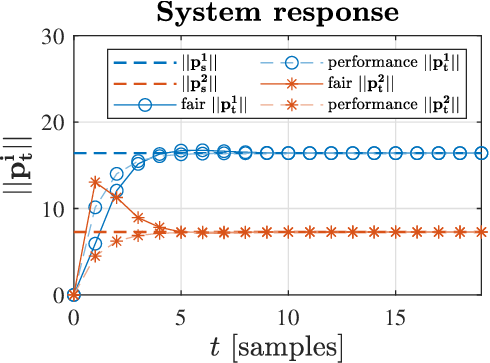} \\
		\end{tabular}}\\
		\subfigure[Motion on the bi-dimensional plane]{\includegraphics[width=0.4\columnwidth]{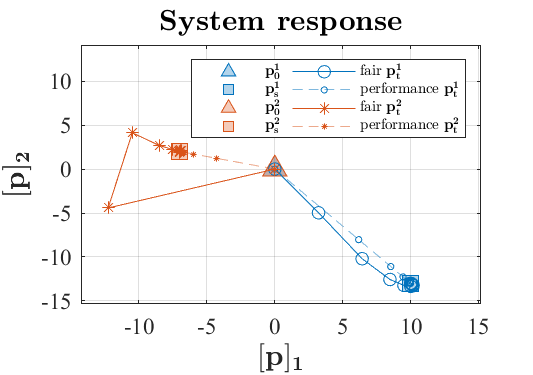}}
	\end{tabular}
	\caption{Bi-dimensional motion in closed-loop (two systems with exhaustible resources): strategy penalizing performance only \emph{vs} Fair-MPC.}\label{fig:2d_2systemsEX}
\end{figure}

\subsection{The two-class case} 

We now increase the number of systems up to $N=8$, with the systems' dynamics in \eqref{eq:motion} characterized by the parameters in \tablename{~\ref{tab:2d_2class_param}}. These systems can be distinguished into two classes based on their capabilities of exploiting resources (\emph{i.e.,} the values of the components of $\mathbf{b}_{i}$) and their distance to the target\footnote{The initial position of all systems is still located at the origin of the plane.}. Within this setting, we impose $\bar{U}=200$ and assume the availability of unlimited control resources. Note that, we design Fair-MPC using the formulation in \eqref{eq:class_based}, by lowering $L$ to $10$, to reduce the computational time. 
\begin{table}[!hp]
	\caption{Bi-dimensional motion (two-class case): parameters of the Fair-MPC problem.}\label{tab:2d_2class_param}
	\centering
	\begin{tabular}{ccccccc}
		$L$ & $Q$ & $\gamma_{u}$ & $\Gamma_{e}$& $\beta$ & $\lambda_x$& $\lambda_u$\\
		\hline 
		\hline
		10 & 1 & $10^{-1}$ & $10^{1}$ & 0.1& 0.1 & 0.1 \\
		\hline
	\end{tabular}
\end{table}
\begin{table}[!hp]
	\caption{Bi-dimensional motion (two-class case): systems' features and targets.}\label{tab:2d_2class_data}
	\centering
	\begin{tabular}{|c|c|c|c|}
		\cline{2-4}
		\multicolumn{1}{c|}{} & $\mathbf{b}^{i}$ & $\mathbf{p}_{s}^i$ & Class\\
		\hline
		system 1 & $\begin{bmatrix}
			0.2 & 0.2
		\end{bmatrix}$ & $\begin{bmatrix}
			10 & -13
		\end{bmatrix}$ & refrained ($C_1$)\\
		\hline
		system 2  & $\begin{bmatrix}
			0.19 & 0.19
		\end{bmatrix}$ &$\begin{bmatrix}
			-7 & 2
		\end{bmatrix}$ & refrained ($C_1$)\\
		\hline
				system 3  & $\begin{bmatrix}
			0.194 & 0.194
		\end{bmatrix}$ &$\begin{bmatrix}
			6 & 3
		\end{bmatrix}$ & refrained ($C_1$)\\
		\hline
				system 4  & $\begin{bmatrix}
			0.186 & 0.186
		\end{bmatrix}$ &$\begin{bmatrix}
			 8 & -4
		\end{bmatrix}$ & refrained ($C_1$)\\
		\hline
				system 5  & $\begin{bmatrix}
			1 & 1
		\end{bmatrix}$ &$\begin{bmatrix}
			2 & -5
		\end{bmatrix}$ & influenced ($C_2$)\\
		\hline
				system 6  & $\begin{bmatrix}
			0.95 & 0.95
		\end{bmatrix}$ &$\begin{bmatrix}
			1 & 2
		\end{bmatrix}$ & influenced ($C_2$)\\
		\hline
				system 7  & $\begin{bmatrix}
			0.97 & 0.97
		\end{bmatrix}$ &$\begin{bmatrix}
			-10 & -13
		\end{bmatrix}$ & influenced ($C_2$)\\
		\hline
				system 8  & $\begin{bmatrix}
			0.93 & 0.93
		\end{bmatrix}$ &$\begin{bmatrix}
			-4 & -1
		\end{bmatrix}$ & influenced ($C_2$)\\
		\hline
	\end{tabular}
\end{table}
The performance of Fair-MPC over a simulation horizon of $T=20$ steps is compared to that of performance-only strategies in \figurename{~\ref{fig:2d_2class}}. Also in this scenario, the effects of introducing fairness into control design is evident. First of all, Fair-MPC leads to a more efficient use of the available budget, especially at the beginning of the simulation horizon. In particular, considering fairness allows us to allocate more control resources to the systems belonging to the \emph{refrained} cluster, while enabling all the systems to attain their tasks at the same time. This result was instead more difficult to achieve with performance-only Fair-MPC, due to the inherent disparities among systems. Note that, to achieve fairness the trajectories of systems closer to their target are slightly deviated also in this scenario. The gain in fairness is thus paired with a sight deterioration in the promptness of target tracking, as proven by the indexes reported in \tablename{~\ref{tab:2d_2class_KPI}}.  
\begin{figure}[!hp]
	\centering
	\begin{tabular}{c}
		\subfigure[1-norm of the inputs for the two classes]{\hspace*{-.5cm}\begin{tabular}{cc}
				\includegraphics[width=0.465\columnwidth]{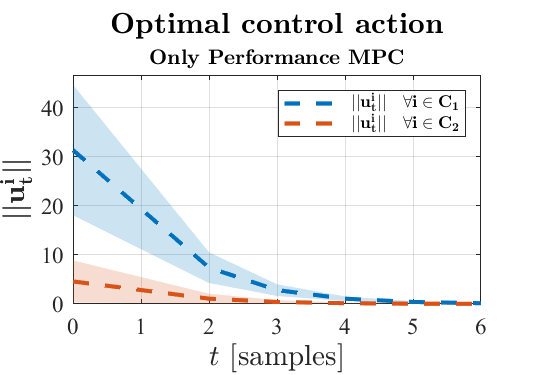} &		\includegraphics[width=0.465\columnwidth]{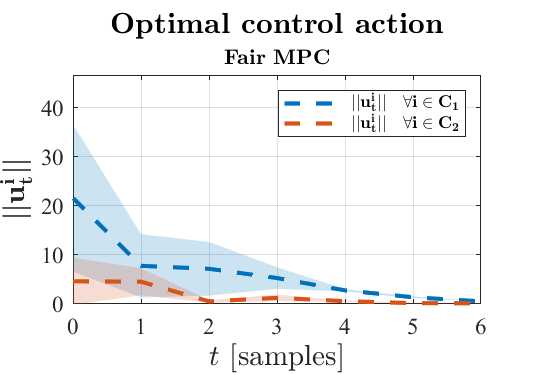} \\
		\end{tabular}}\\
		\subfigure[1-norm of the positions for the two classes]{\hspace*{-.5cm}\begin{tabular}{cc}
				\includegraphics[width=0.465\columnwidth]{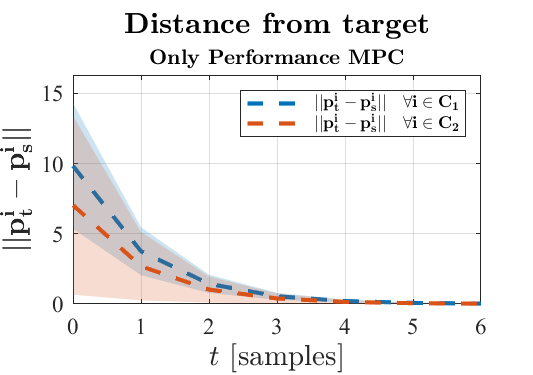} &		\includegraphics[width=0.465\columnwidth]{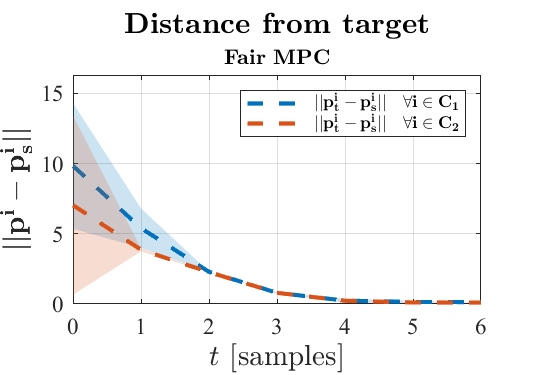} \\
		\end{tabular}}\\
		\subfigure[Motion on the bi-dimensional plane]{\hspace*{-.5cm}\begin{tabular}{cc}
		\includegraphics[width=0.41\columnwidth]{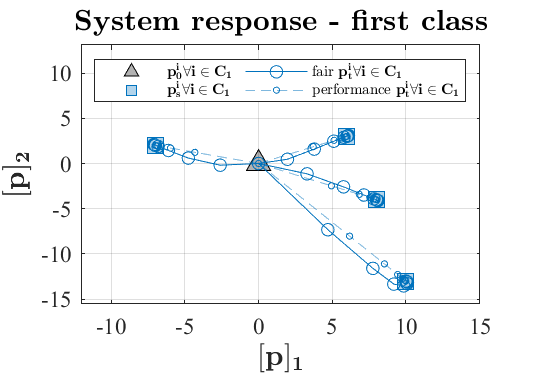} \hspace*{-.3cm}&\hspace*{-.3cm}		\includegraphics[width=0.41\columnwidth]{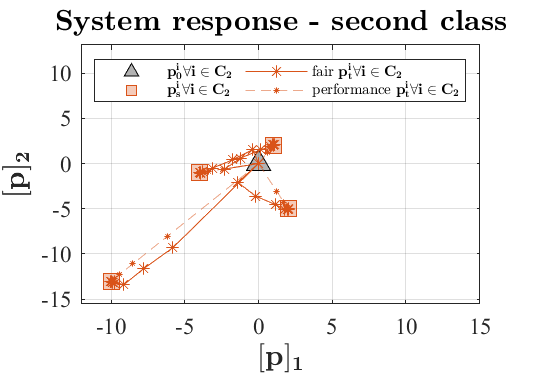} \\
	\end{tabular}}
	\end{tabular}
	\caption{Bi-dimensional motion in closed-loop (two-class): strategy penalizing performance only\emph{vs} Fair-MPC. In the first two blocks of plots, we report the mean (dashed line) and standard deviation (shaded area) of the inputs' and positions' 1-norm.}\label{fig:2d_2class}
\end{figure}
\begin{table}[!hp]
	\caption{Bi-dimensional motion in closed-loop (two-class case): performance indexes with $\alpha_{\%}=10\%$ in \eqref{eq:tau} \emph{vs} strategy.}\label{tab:2d_2class_KPI}
	\centering
	\begin{tabular}{|c|c|c|c|c|}
		\cline{2-5}
		\multicolumn{1}{c|}{} & $\mathcal{H}_{s}$ \eqref{eq:tracking_index2} & $\mathcal{H}_{\tau}$ \eqref{eq:tracking_time}& $\mathcal{H}_{u}$ \eqref{eq:equality_index} & $\mathcal{H}_{e}$ \eqref{eq:equity_index}\\
		\hline
		Performance only &1.000 & \textbf{0.800} & 0.499 & 0.676\\
		\hline
		Fair-MPC &1.000 & 0.768 & \textbf{0.515} & \textbf{0.733} \\
		\hline
	\end{tabular}
\end{table}
\section{Conclusions}
We have presented Fair-MPC, a preliminary approach to explicitly account for fairness in constrained control design. Toward this objective, we have formulated a performance-oriented predictive control problem, that also promotes equity and equality among a group of systems. Based on the peculiarities of this formulation, we have additionally proposed a tailored strategy to tune the fairness-related penalties in the Fair-MPC cost, thus alleviating the practical calibration burden of the proposed approach. The presented numerical examples have spotlighted the impact of fairness on both individual performance and on the distribution of the available control resources. Overall, we can conclude that stronger systems in the group can be subject to a slight loss in tracking performance, but this is deemed to be generally a minor limitation since a more fair closed-loop behavior is attained. Indeed, by using Fair-MPC, we have generally observed a more balanced behavior of the group and a more similar control effort required by all systems.    

Future research will be devoted to extending this preliminary formulation of Fair-MPC to more general challenging scenarios, while analyzing its additional properties and the conservativeness of the current formulation. We also aim at distributing Fair-MPC to reduce the computational burden on the central unit, while having a scheme that is resilient to centralized failures. Future work will also be focused on devising strategies to soften the requirements on equality after at least one system has reached its target, for the remaining control resources to be actually allocated to the systems that have been left behind.  

\section{Acknowledgments}
The work was partially supported by the following: Italian Ministry of Enterprises and Made in Italy in the framework of the project 4DDS (4D Drone Swarms) under grant no. F/310097/01-04/X56,  PRIN project TECHIE: “A control and network-based approach for fostering the adoption of new technologies in the ecological transition” Cod. 2022KPHA24 CUP: D53D23001320006 and by MOST – Sustainable Mobility National Research Center that received funding from the European Union NextGenerationEU (PIANO NAZIONALE DI RIPRESA E RESILIENZA (PNRR) – MISSIONE 4 COMPONENTE 2, INVESTIMENTO 1.4-D.D. 1033 17/06/2022, CN00000023), Spoke 5 “Light Vehicle and Active Mobility”.

\bibliographystyle{abbrv} 
\bibliography{main.bib}

\end{document}